\documentclass{amsart}
\usepackage{amssymb, amsmath, latexsym, mathabx, dsfont,tikz}
\usepackage{multirow}
\usepackage{setspace}
\usepackage{enumerate}
\usepackage{diagbox}

\renewcommand{\baselinestretch}{\baselinestretch}
\renewcommand{\baselinestretch}{1.1}
\numberwithin{equation}{section}

\newtheorem{thm}{Theorem}[section]
\newtheorem{lem}[thm]{Lemma}

\newtheorem{prop}[thm]{Proposition}

\theoremstyle{definition}

\theoremstyle{remark}
\newtheorem{rmk}[thm]{Remark}

\numberwithin{equation}{section}

\newcommand{\gen}{\text{gen}}
\newcommand{\spn}{\text{spn}}
\newcommand{\ord}{\text{ord}}

\newcommand{\n}{{\mathbb N}}
\newcommand{\z}{{\mathbb Z}}
\newcommand{\q}{{\mathbb Q}}

\newcommand{\D}{{\Delta}}
\newcommand{\Mod}[1]{\ (\mathrm{mod}\ #1)}

\begin{document}
\title[]{Regular ternary triangular forms}

\author{Mingyu Kim and Byeong-Kweon Oh}

\address{Department of Mathematical Sciences, Seoul National University, Seoul 151-747, Korea}
\email{kmg2562@snu.ac.kr}

\address{Department of Mathematical Sciences and Research Institute of Mathematics, Seoul National University, Seoul 151-747, Korea}
\email{bkoh@snu.ac.kr}

\thanks{This work of the second author was supported by the National Research Foundation of Korea (NRF-2017R1A2B4003758).}



\begin{abstract} 
An integer of the form $T_x=\frac{x(x+1)}2$ for some positive integer $x$ is called a triangular number. A ternary triangular form $aT_{x}+bT_{y}+cT_{z}$ for positive integers $a,b$ and $c$ is called regular if it represents every positive integer that is locally represented. 
In this article, we prove that there are exactly 49 primitive regular ternary triangular forms.
\end{abstract}

\maketitle

\section{Introduction}

A quadratic homogeneous polynomial 
$$
f(x_1,x_2,\dots,x_n)=\sum_{i,j=1}^n a_{ij} x_ix_j \quad (a_{ij}=a_{ji} \in \z)
$$
is called an integral quadratic form. Throughout this article, we always assume that $f$ is positive definite, that is,  $f(x_1,x_2,\dots,x_n)>0$ for any non-zero real vector $(x_1,x_2,\dots,x_n) \in \mathbb R^n$. 
Let $R$ be any ring containing $\z$. For an integer $N$, if there is a solution $(x_1,x_2,\dots,x_n) \in R^n$ of the equation $f(x_1,x_2,\dots,x_n)=N$, then we say that $N$ is represented by $f$ over $R$. In particular, if $N$ is represented by $f$ over the $p$-adic integer ring $\z_p$ for any prime $p$, then we say that $N$ is locally represented by $f$. From the definition, note that any integer that is represented by $f$ over $\z$ is locally represented by $f$. However, it is known that the converse is not true, in general. A positive definite integral quadratic form is called {\it regular} if the converse is also true, that is, it represents every integer over $\z$ that is locally represented. 

Dickson \cite{D} who initiated the study of regular quadratic forms first coined the term regular. 
Jones and Pall \cite{JP} gave the list of all 102 primitive diagonal regular ternary quadratic forms. 
Watson proved in his thesis \cite{W1} that there are only finitely many equivalence classes of primitive positive definite ternary regular forms. 
Jagy, Kaplansky and Schiemann \cite{JKS} succeeded Watson's study on regular quadratic forms and provide the list of 913 candidates of regular positive definite integral ternary quadratic forms. All but 22 of them are already proved to be regular at that time. 
Recently, the second author \cite{O1} proved the regularities of 8 ternary quadratic forms among remaining 22 candidates.
A conditional proof for the remaining 14 candidates under the Generalized Riemann Hypothesis was given by Lemke Oliver \cite{L}. Note that there are infinitely many regular positive definite integral quaternary quadratic forms (for this, see \cite{E95}).

Now we look into {\it the representations of ternary triangular forms}. An integer of the form $T_x=\frac{x(x+1)}{2}$ for some  positive integer $x$ is called {\it a triangular number}. For positive integers $a_1,a_2,a_3$ with $a_1\le a_2\le a_3$, a polynomial of the form 
{\small$$\D(a_1,a_2,a_3)=\D(a_1,a_2,a_3)(x_1,x_2,x_3):=a_1\frac{x_1(x_1+1)}{2}+a_2\frac{x_2(x_2+1)}{2}+a_3\frac{x_3(x_3+1)}{2}$$}
is called {\it a  ternary triangular form}. We say an integer $N$ is represented by the triangular form $\D(a_1,a_2,a_3)$ if 
\begin{equation} \label{1}
\D(a_1,a_2,a_3)(x_1,x_2,x_3)=N
\end{equation}
has an integer solution $(x_1,x_2,x_3) \in \z^3$. Similarly to the quadratic form case, if Equation \eqref{1} has a solution in $\z_p$ for any prime $p$, then we say that $N$ is locally represented by $\D(a_1,a_2,a_3)$. We say  $\D(a_1,a_2,a_3)$ is {\it regular} if it represents every integer that is locally represented.  Note that any ternary triangular form is regular if it represents all positive integers. Such a ternary triangular form is said to be {\it  universal}.

Gauss' Eureka Theorem says that every positive integer is a sum of at most three triangular numbers, which implies that
the ternary triangular form $\D(1,1,1)$ is universal.
In 1862, Liouville classified all universal ternary  triangular forms, and they are, in fact,  the following seven forms:
$$
\D(1,1,1), \ \ \D(1,1,2), \ \ \D(1,1,4), \ \ \D(1,1,5), \ \
\D(1,2,2), \ \ \D(1,2,3), \ \ \D(1,2,4).
$$
As mentioned above, these universal triangular forms are regular. In 2013, Chan and Oh \cite{CO13} proved that there are only finitely many  regular ternary triangular forms.
In 2015, Chan and Ricci \cite{CR} proved the finiteness of regular ternary triangular forms in a more general setting. 
They actually proved that for any given positive integer $c$, there are only finitely many inequivalent positive ternary regular primitive complete quadratic polynomials with conductor $c$. 
From this follows the finiteness of regular ternary $m$-gonal forms. 
Note that an integer of the form $\frac {(m-2)x^2-(m-4)x}2$ for some integer $x$ is called an $m$-gonal number, and a (regular) ternary  $m$-gonal form is defined similarly.

In this article, we prove that there are exactly $49$ regular ternary triangular forms. 
In the previous papers \cite{CO13} and \cite{CR}, the authors use Burgess' estimation on character sums (for this, see \cite{Bu} and \cite{E}) to prove the finiteness of regular ternary triangular forms. 
It seems to be quite difficult to find an explicit upper bound of the discriminant of regular ternary triangular forms by using Burgess' estimation. In this article, we use a purely arithmetic method to find such an explicit and effective upper bound of the discriminant of regular ternary triangular forms, and finally, we classify all regular ternary triangular forms.

  A $\z$-lattice $L$ is a finitely generated free $\z$-module equipped with a non-degenerate symmetric bilinear form $B$ such that
$B(L,L) \subset \z$. The corresponding quadratic map $Q$ is defined by $Q(\mathbf v)=B(\mathbf v,\mathbf v)$ for any $\mathbf v \in L$. 

Let  $L = \z \mathbf x_1 + \z \mathbf x_2 + \cdots + \z \mathbf x_n$ be a $\z$-lattice. The quadratic form $f_L$ corresponding to $L$ is defined by 
$f_L(x_1,x_2,\dots,x_n)=\sum B(\mathbf x_i,\mathbf x_j)x_ix_j$. Furthermore, the corresponding symmetric matrix $M_L$ is defined by 
$M_L=(B(\mathbf x_i,\mathbf x_j))$, which is called the {\it matrix presentation} of $L$. If $L$ admits an orthogonal basis $\{\mathbf x_1, \dots , \mathbf x_n \}$, we call $L$ {\it diagonal} and simply write
$$
L=\langle Q(\mathbf x_1), \dots , Q(\mathbf x_n) \rangle.
$$
For any odd prime $p$, $\D_p$ denotes a non-square unit in $\z_p$. 

Any unexplained notations and terminologies can be found in \cite{Ki} or \cite{OM}.

\section{Preliminaries}

A nonnegative integer of the form $T_x=\frac{x(x+1)}{2}$ for some positive integer $x$ is called  \textit{a triangular number}. For example, $0,1,3,6,10,15,\cdots$ are triangular numbers. Since $T_x=T_{1-x}$, $T_x$ is a triangular number for any integer $x$.  
For positive integers $a_1,a_2,\dots,a_k$ with $a_1 \le a_2\le \dots\le a_k$, we call a polynomial of the form
$$
\D(a_1,a_2,\dots,a_k)=\D(a_1,a_2,\dots,a_k)(x_1,x_2,\dots,x_k)=a_1T_{x_1}+a_2T_{x_2}+\cdots+a_kT_{x_k}
$$
 \textit{a $k$-ary triangular form.} 
For a triangular form $\D=\D(a_1,a_2,\dots,a_k)$,  we define $d(\D(a_1,a_2,\dots,a_k))=a_1a_2\cdots a_k$, which is called the discriminant of the triangular form $\Delta(a_1,a_2,\dots,a_k)$.  
A triangular form $\D(a_1,a_2,\dots,a_k)$ is called \textit{primitive} if $\gcd(a_1,a_2,\dots,a_k)=1$. 
Unless stated otherwise, we always assume that
\begin{center}
{\textit{every triangular form is primitive.}}
\end{center}
For an integer $n$ and a $k$-ary triangular form $\D(a_1,a_2,\dots,a_k)$, we say that $n$ \textit{is represented by} $\D(a_1,a_2,\dots,a_k)$ if the Diophantine equation
$$
a_1T_{x_1}+a_2T_{x_2}+\cdots+a_kT_{x_k}=n
$$
has an integral solution. In this case, we write  $n\longrightarrow \D(a_1,a_2,\dots,a_k)$.
We also define
$$
T(n,\langle a_1,a_2,\dots,a_k \rangle)
=\left\{ (z_1,z_2,\dots,z_k)\in \z^k : a_1T_{z_1}+a_2T_{z_2}+\cdots+a_kT_{z_k}=n \right\}
$$
and $t(n,\langle a_1,a_2,\dots,a_k \rangle)$ to be the cardinality of the above set.

A triangular form $\D(a_1,a_2,\dots,a_k)$ is called \textit{universal} 
if it represents every positive integer, that is,
$$
a_1T_{x_1}+a_2T_{x_2}+\cdots+a_kT_{x_k}=n \text{ is soluble in } \z
$$
for any positive integer $n$.
A triangular form $\D(a_1,a_2,\dots,a_k)$ is called \textit{regular} 
if it globally represents every integer which is locally represented. 
In other words, $\D(a_1,a_2,\dots,a_k)$ is regular 
if for any integer $n$ such that $a_1T_{x_1}+a_2T_{x_2}+\cdots+a_kT_{x_k}=n$ is soluble in $\z_p$ for any prime $p$, 
the diophantine equation  $a_1T_{x_1}+a_2T_{x_2}+\cdots+a_kT_{x_k}=n$ is soluble in $\z$. As shown in \cite{CO13}, any primitive triangular form is universal over $\z_2$.

Note that a triangular form $\D(a_1,a_2,\dots,a_k)$ represents $n$ if and only if the Diophantine equation
$$
a_1(2x_1-1)^2+a_2(2x_2-1)^2+\cdots+a_k(2x_k-1)^2=8n+a_1+a_2+\cdots+a_k
$$
is soluble in $\z$.
This equivalence shows how the representation of a triangular form is transformed into the representation of a diagonal quadratic form with congruence conditions.
Now, we can reformulate the regularity  in a practical way. 
A ternary triangular form $\D(a,b,c)$ is regular if the following implication holds: for any positive integer $n$,
if $ax^2+by^2+cz^2=8n+a+b+c$ is soluble in $\z_p$ for any odd prime $p$, then there exist odd integers $x,y$ and $z$ such that $ax^2+by^2+cz^2=8n+a+b+c$.

Let $f(x_1,x_2,\dots,x_k)$ be a positive definite integral quadratic form of rank $k$ and let $n$ be an integer. We define
$$
R(n,f)=\left\{ (x_1,x_2,\dots,x_k)\in \z^k : f(x_1,x_2,\dots,x_k)=n\right\} \ \ \text{and} \ \ r(n,f)=\vert R(n,f)\vert .
$$
We say that $n$ \textit{is represented by} $f$ if $r(n,f)>0$.
For a vector $\mathbf{d}=(d_1,d_2,\dots,d_k)\in (\z /2\z)^k$, we also define
$$
R_{\mathbf{d}}(n,f)=\left\{ (x_1,x_2,\dots,x_k)\in R(n,f) : x_i\equiv d_i\Mod 2 \ \ \text{for any} \ \ 1\le i\le k \right\}
$$
and $r_{\mathbf{d}}(n,f)=\vert R_{\mathbf{d}}(n,f)\vert$.

For an integer $n$ and a diagonal quadratic form $\langle a_1,a_2,\dots,a_k\rangle$, we write
$$
n\overset{2}{\longrightarrow} \langle a_1,a_2,\dots,a_k\rangle
$$
if there is a vector $(x_1,x_2,\dots,x_k)\in \z^n$ with 
$(x_1x_2\cdots x_k,2)=1$ such that $a_1x_1^2+a_2x_2^2+\cdots+a_kx_k^2=n$. 
We also use the notation
$$
n\overset{2}{\nrightarrow} \langle a_1,a_2,\dots,a_k\rangle
$$
if there does not exist such a vector  $(x_1,x_2,\dots,x_k)\in \z^n$. Under these notations, the followings are all equivalent:
\begin{enumerate} [(i)]
\item $n\longrightarrow \D(a_1,a_2,\dots,a_k)$;
\item $t(n,\langle a_1,a_2,\dots,a_k\rangle)>0$;
\item $r_{(1,1,\dots,1)}(8n+a_1+a_2+\cdots+a_k,\langle a_1,a_2,\dots,a_k\rangle)>0$;
\item $8n+a_1+a_2+\cdots+a_k\overset{2}{\longrightarrow} \langle a_1,a_2,\dots,a_k\rangle$.
\end{enumerate}

Let $L$ be a $\z$-lattice and let $m$ be a positive integer.
\textit{Watson transformation of $L$ modulo $m$} is defined by
$$
\Lambda_m(L)=\{ x\in L : Q(x+z)\equiv Q(z) \Mod m \ \text{for any} \ z\in L \}. 
$$
We denote by $\lambda_m(L)$ the primitive $\z$-lattice obtained from 
$\Lambda_m(L)$ by scaling $L\otimes \q$ by a suitable rational number.
Let $p$ be an odd prime. Let $L=\langle a,p^mb,p^nc\rangle$ be a ternary 
$\z$-lattice, where $(abc,p)=1$ and $0\le m\le n$. 
Then one may easily check 
$$
\lambda_p(L)\simeq \begin{cases} \langle a,b,c\rangle & \text{if } m=n=0, \\[5pt]
\langle pa,b,p^{n-1}c \rangle & \text{if } 1=m\le n, \\[5pt]
\langle a,p^{m-2}b,p^{n-2}c \rangle & \text{if } 1<m\le n. \end{cases}
$$
For a ternary triangular form $\D(a,b,c)$ and an odd prime $p$, we define
$$
\lambda_p(\D(a,b,c))=\D(a',b',c'),
$$
where $\langle a',b',c'\rangle \simeq \lambda_p(\langle a,b,c\rangle)$.

\begin{lem} \label{lemdescend1}
Let $p$ be an odd prime and let $a,b,c$ be positive integers which are not divisible by $p$.
Let $r, s$ be positive integers.
If the ternary triangular form $\D(a,p^rb,p^sc)$ is regular, then so is $\lambda_p(\D(a,p^rb,p^sc))$.
\end{lem}

\begin{proof} See \cite[Lemma 3.3]{CO13}.
\end{proof}

Though the proof of the next lemma is quite similar to the proof of Lemma \ref{lemdescend1}, we provide the proof for completeness.

\begin{lem} \label{lemdescend2}
Let $p$ be an odd prime and let $s$ be a positive integer. Let $a,b$, and $c$ be positive integers such that 
$(p,abc)=1$ and $\left( \frac{-ab}{p}\right) =-1$, where $\left( \frac{\cdot}{p} \right)$ is the Legendre symbol modulo $p$.
If the ternary triangular form $\Delta(a,b,p^sc)$ is regular, 
then so is $\lambda_p(\Delta(a,b,p^sc))$.
\end{lem}

\begin{proof}
It is enough to show that $\Delta(p^2a,p^2b,p^sc)$ is regular. 
Let $n$ be a positive integer such that the equation
\begin{equation} \label{eq1}
p^2aT_x+p^2bT_y+p^scT_z=n
\end{equation}
is soluble in $\z_p$ for any prime $p$. Then 
$$
8n+p^2a+p^2b+p^sc\longrightarrow \gen(\langle p^2a,p^2b,p^sc \rangle).
$$ 
Thus
$$
8\left(n+\frac{p^2-1}{8}a+\frac{p^2-1}{8}b\right)+a+b+p^sc
\longrightarrow \gen(\langle a,b,p^sc \rangle).
$$
Since $\Delta(a,b,p^sc)$ is regular, there is a vector $(x,y,z)\in \z^3$ with 
$xyz\equiv 1\Mod 2$ such that $ax^2+by^2+p^scz^2=8n+p^2a+p^2b+p^sc$. 
Since $n$ is divisible by $p$, we have $ax^2+by^2\equiv 0\Mod p$. 
From the assumption $\left(\displaystyle\frac{-ab}{p}\right)=-1$, 
we have $x\equiv y\equiv 0\Mod p$. So
$$
p^2a\left(\frac{x}{p}\right)^2+p^2b\left(\frac{y}{p}\right)^2+p^scz^2
=8n+p^2a+p^2b+p^sc
$$
with $\frac{x}{p} \cdot \frac{y}{p} \cdot z \equiv 1\Mod 2$. 
Thus Equation $\eqref{eq1}$ is soluble in $\z$. This completes the proof.
\end{proof}

For an odd prime $p$ and a ternary $\z$-lattice $L$, we say that $L$ is \textit{$p$-stable} if
$$
\langle 1,-1\rangle \longrightarrow L_p \quad \text{or} \quad L_p\simeq \langle 1,-\D_p\rangle \perp \langle p\epsilon_p\rangle
$$
for some $\epsilon_p\in \z_p^{\times}$. We say that $L$ is \textit{stable} if $L$ is $p$-stable for every odd prime $p$.
A ternary triangular form is called \textit{$p$-stable} (\textit{stable}) if the corresponding quadratic form is $p$-stable (stable, respectively).
Let $\D(a,b,c)$ be a regular ternary triangular form. 
Then by taking $\lambda_q$-transformations to $\D(a,b,c)$ repeatedly, if possible, for any odd prime $q$ dividing the discriminant,  we may obtain a stable regular ternary triangular form $\D(a',b',c')$ by Lemmas \ref{lemdescend1} and \ref{lemdescend2}.
Note that the corresponding quadratic form $\langle a',b',c'\rangle$ has a smaller discriminant and a simpler local structure than $\langle a,b,c\rangle$.

\section{Stable regular ternary triangular forms}
In this section, we prove that there are exactly 17 stable regular ternary triangular forms.
Throughout this section, $r_k$ denotes the $k$-th odd prime so that $\{ r_1=3<r_2=5<r_3=7<\cdots \}$ is the set of all odd primes.  Let $\D(a,b,c)$ be a stable regular ternary triangular form. 
We always assume that $0<a\le b\le c$.

\begin{lem} \label{lembound1}
For an integer $s$ greater than $1$, let $p_1<p_2<\cdots <p_s$ be odd primes. 
Let $u$ be an integer with $(u,p_1p_2\cdots p_s)=1$ and let $v$ be an arbitrary integer.
Then there is an integer $n$ with $0\le n<(s+2)2^{s-1}$ such that 
$(un+v,p_1p_2\cdots p_s)=1$.
\end{lem}

\begin{proof}
See \cite[Lemma 3]{KKO}.
\end{proof}

Though Lemma \ref{lembound1} gives, in general, a nice upper bound of the longitude of arithmetic progression satisfying the assumption,  there is a shaper bound in some restricted situation.

\begin{lem} \label{lembound2}
Under the same notations given in Lemma \ref{lembound1}, if  $s<p_1$, 
then there is an integer $n$ with $0\le n\le s$ such that $(un+v,p_1p_2\cdots p_s)=1$.
\end{lem}

\begin{proof}
Trivial.
\end{proof}

\begin{lem} \label{lembound3}
Let $p\ge 5$ be a prime and let $d$ be a positive integer with $(d,p)=1$. 
Let $L=\langle a,b,c \rangle$ be a $p$-stable $\z$-lattice that is anisotropic over $\z_p$. 
Then there is an integer $g$ such that
\begin{itemize} 
\item [(i)] $0<g<p^2$;
\item [(ii)] $dg+a+b\nrightarrow \langle a,b \rangle$ over $\z_p$;
\item [(iii)] $dg+a+b+c\longrightarrow \langle a,b,c \rangle$ over $\z_p$;
\item [(iv)] $\text{max}\left\{ \ord_p(dg+a+b),\ord_p(dg+a+b+c)\right\} \le 1$.
\end{itemize}
\end{lem}

\begin{proof}
Since $L$ is $p$-stable and is anisotropic over $\z_p$ by assumption, we have 
$$
\langle a,b,c \rangle \simeq \langle 1,-\Delta_p \rangle \perp \langle p\epsilon_p \rangle \text{ over } \z_p,
$$
for some $\epsilon_p\in \z_p^{\times}$.  First, we assume that $p$ divides $c$. 
Since $\langle a,b \rangle \simeq \langle 1,-\Delta_p \rangle$, it does not represent $\gamma \in \z_p$ satisfying $\ord_p\gamma \equiv 1\Mod 2$. 
Since $p\ge 5$, there exists a positive integer $g_1$ with $g_1<p^2$ such that
$$
dg_1+a+b\equiv 3c\Mod{p^2}.
$$
Then one may easily check that $g_1$ satisfies all conditions given above. 
Now, assume that $p$ divides $ab$. Without loss of generality, we may assume that $p$ divides $b$. 
Since $p\ge 5$, there exists an integer $a'$ with $(p,a')=1$ such that 
$aa'$ is not a square modulo $p$ and $a'\not\equiv -c\Mod p$. 
We take a positive integer $g_2$ with $g_2<p$ such that $dg_2+a+b\equiv a'\Mod p$. 
One may easily show that $g_2$ satisfies all conditions given above, which completes the proof.
\end{proof}

Let $T$ be the set of odd primes $p$ such that the diagonal ternary quadratic form $\langle a,b,c\rangle$ is anisotropic over $\z_p$. Since such primes are only finitely many, we let
\begin{align*}
T&=\{ p : p\ge 3,\; \langle a,b,c \rangle \text{ is anisotropic over } \z_p \} \\
&=\{ p_1<p_2<\cdots <p_t \} .
\end{align*}
Let
$$
T'=T-\{3\}=\{ q_1<q_2<\dots<q_{t'} \}.
$$
Note that $t'=t$ if $3\not\in T$, and $t'=t-1$ otherwise.

\begin{lem} \label{lemt17}
Under the assumptions given above, we have $t'\le 17$.  
\end{lem}

\begin{proof}
Note that $\langle a,b,c\rangle$ represents every integer of the form $24n+a+b+c$ over $\z_3$.
Let $g$ be a positive integer satisfying Lemma \ref{lembound3} in the case when $p=q_1$ and $d=24$.

By Lemma \ref{lembound1}, there is an integer $h$ with $0\le h<(t'+1)2^{t'-2}$ 
such that $(24q_1^2 h+24g+a+b+c,q_2\cdots q_{t'})=1$.
If we let $k=q_1^2 h+g$, then one may easily show that
\begin{equation} \label{eq2}
24k+a+b\nrightarrow \langle a,b \rangle
\end{equation}
and
$$
24k+a+b+c\longrightarrow \gen(\langle a,b,c \rangle ).
$$ 
Since $\Delta (a,b,c)$ is regular, 
there is a vector $(x,y,z)\in \z^3$ with $xyz\equiv 1\Mod 2$ 
such that $ax^2+by^2+cz^2=24k+a+b+c$.
From Equation \eqref{eq2}, we have $z^2\ge 9$. 
So $a+b+9c\le 24k+a+b+c$ and we have $c\le 3k$. Now
$$
q_1q_2\cdots q_{t'}\le abc\le c^3\le (3k)^3\le (3q_1^2(t'+1)2^{t'-2})^3.
$$
Assume to the contrary that $t'\ge 18$. Then one may easily show that 
$$
r_8r_9\cdots r_{t'+1}>(3(t'+1)2^{t'-2})^3.
$$ 
Since $q_i\ge r_{i+1}$ for any $i$, we have
$$
(q_1\cdots q_6)q_7q_8\cdots q_{t'} > q_1^6 \cdot r_8r_9\cdots r_{t'+1}>(3q_1^2(t'+1)2^{t'-2})^3,
$$
which is a contradiction. Therefore we have $t'\le 17$. This completes the proof. 
\end{proof}

If we are able to use Lemma  \ref{lembound2} instead of Lemma \ref{lembound1}, then we may have more effective upper bound of $t'$ than the previous lemma.  

\begin{lem} \label{keylem} Under the same notations given above, if $0<t'-j<q_{j+1}$ for some $j$ such that $1\le j\le t'-1$, then we have
$$
q_1q_2\cdots q_{t'}<a(3q_1^2q_2\cdots q_j(t'-j+1))^2\le (3q_1^2q_2\cdots q_j(t'-j+1))^3.
$$ 
\end{lem}

\begin{proof}
Note that $\langle a,b,c \rangle$ represents every integer of the form $24n+a+b+c$ over $\z_3$.
Let $g$ be a positive integer satisfying Lemma \ref{lembound3} in the case when $p=q_1$ and 
$d=24$. Let 
$$
g_j=\begin{cases} g \quad &\text{if $j=1$}, \\
g+\epsilon_1q_1^2 \quad &\text{if $j=2$}, \\
g+\epsilon_1 q_1^2+\epsilon_2 q_1^2q_2+\cdots+\epsilon_{j-1}q_1^2q_{2}q_{3}\cdots q_{j-1} \quad &\text{if $j\ge 3$},
\end{cases}
$$
where for each $i$, $\epsilon_i$ is suitably chosen in $\{0,1\}$ so that 
$$
24g_j+a+b+c \not\equiv 0 \Mod {q_2\cdots q_j}
$$
for any $j\ge 2$.
Note that
$g_1=g<q_1^2$ and 
$g_j < q_1^2q_2\cdots q_{j}$ for any $j\ge 2$. Since $0<t'-j<q_{j+1}$ by assumption, we apply Lemma \ref{lembound2} with odd primes
$q_{j+1}<q_{j+2}<\cdots<q_{t'}$, $u=24q_1^2 q_2\cdots q_j$ and $v=24g_j+a+b+c$ so that we may conclude that there is an integer $s$ with $0\le s\le t'-j$ such that 
$$
(24q_1^2 q_2\cdots q_js+24g_j+a+b+c,q_{j+1}q_{j+2}\cdots q_{t'})=1.
$$ 
Therefore, by a similar reasoning to Lemma \ref{lemt17}, we have $c\le 3q_1^2q_2\cdots q_j(t'-j+1)$.  
The lemma follows directly from this. 
\end{proof}

\begin{lem} \label{lemt10}
Under the assumptions given above, we have $t\le 10$.  
\end{lem}

\begin{proof} By Lemma \ref{lemt17}, we may assume that $t'\le 17$. 
First, assume that $q_1 \ge 13$. Since $t'-1<17\le q_2$, 
we may apply Lemma \ref{keylem} so that
$$
q_1q_2\cdots q_{t'}<(3q_1^2 t')^3.
$$
From this, one may easily show that $t'\le 8$. 

Now, assume that $q_1=11$. Since $t'-2<17\le q_3$, we may apply Lemma \ref{keylem} so that we may conclude that
$$
q_1q_2\cdots q_{t'}<\left(3q_1^2q_2(t'-1)\right)^3.
$$
Suppose that $t'\ge 11$. Since $r_8=23, r_9=29, r_{10}=31,\dots$, one may directly show that
$$
11\cdot r_8r_9\cdots r_{t'+3}>\left(3\cdot 11^2\cdot (t'-1)\right)^3.
$$
Since $q_i\ge r_{i+3}$ for any $i$, we have
$$
q_1q_2\cdots q_{t'}>11q_2^3 r_8r_9\cdots r_{t'+3}>\left(3\cdot 11^2\cdot q_2\cdot (t'-1)\right)^3,
$$
which is a contradiction. Therefore we have $t'\le 10$.
Now, since $t'-1<13\le q_2$,  we deduce, similarly to the above, that
$$
q_1q_2\cdots q_{t'}<(3q_1^2t')^3,
$$
and thus $t'\le 7$.

Assume that $q_1=7$. Since $t'-3<17\le q_4$ in this case,
one may deduce that
$$
q_1q_2\cdots q_{t'}<\left(3q_1^2q_2q_3(t'-2)\right)^3,
$$
and thus we have $t'\le 12$. Now, since $t'-2<13\le q_3$, we may have
$$
q_1q_2\cdots q_{t'}<\left(3q_1^2q_2(t'-1)\right)^3,
$$
and hence $t'\le 9$.  Since $t'-1<11\le q_2$,  
$$
q_1q_2\cdots q_{t'}<\left(3q_1^2t'\right)^3.
$$
Therefore, we have $t'\le 7$.

Finally, assume that $q_1=5$. Since $t'-4<17\le q_5$, 
we have 
$$
q_1q_2\cdots q_{t'}<\left(3q_1^2q_2q_3q_4(t'-3)\right)^3 \ \  \text{and thus $t'\le 14$.}
$$
Now, since $t'-3<13\le q_4$, we have 
$$
q_1q_2\cdots q_{t'}<\left(3q_1^2q_2q_3(t'-2)\right)^3 \ \  \text{and $t'\le 12$.}
$$ 
Then, since $t'-2<11\le q_3$, we have
$$
q_1q_2\cdots q_{t'}<\left(3q_1^2q_2(t'-1)\right)^3, \    \  \text{and finally we have $t'\le 9$. }
$$
The lemma follows directly from this.
\end{proof}

Recall that we are assuming that $\Delta(a,b,c)$ is stable. Hence for any odd prime $p$,
$$
\langle 1,-1 \rangle \longrightarrow \langle a,b,c \rangle   \ \ \text{over $\z_p$}
\quad \text{or} \quad 
\langle a,b,c \rangle \simeq \langle 1,-\Delta_p \rangle \perp \langle p\epsilon_p \rangle \ \ \text{over $\z_p$},  
$$
for some $\epsilon_p \in \z_p^{\times}$.
In the former case, every element in $\z_p$ is represented by $\langle a,b,c\rangle$ over $\z_p$. In the latter case,
$$
\{\gamma \in \z_p : \gamma \nrightarrow \langle a,b,c\rangle \ \ \text{over} \ \ \z_p \} =\left\{p^{2w-1}\delta_p : w\in \n ,\ \delta_p\in \z_p^{\times},\ \delta_p\epsilon_p\not\in \left(\z_p^{\times}\right)^2\right\} .
$$
Recall that $r_j$ is the $j$-th odd prime. Let $u$ be a positive integer not divisible by $r_j$ and let $v$ be an integer. 
Let $\eta_{r_j} \in \{1,\Delta_{r_j}\}$.  For a positive integer $i$, we define 
$$
\Psi_{u,v}(i,j;\eta_{r_j})=\left\vert \left\{ un+v : 1\le n\le i,\ un+v\nrightarrow \langle 1,-\D_{r_j} \rangle \perp \langle \eta_{r_j}\cdot r_j \rangle \ \ \text{over}\ \ \z_{r_j} \right\} \right\vert.
$$
We also define 
$$
\Psi_{u,v}(i,j)=\max\{\Psi_{u,v}(i,j;1),\Psi_{u,v}(i,j;\Delta_{r_j})\}.
$$ 
 Let $i={b_{e-1}b_{e-2}\dots b_0}_{(r_j)}$ be the base-$r_j$ representation of $i$, that is, 
 $$
 i=b_{e-1}r_j^{e-1}+b_{e-2}r_j^{e-2}+\cdots +b_0
 $$ 
 with $0\le b_{\nu}<r_j$ for $\nu =1,2,\dots, e-1$ and $b_{e-1}>0$. 
We define 
$$
\epsilon_{i,j}(k)=\begin{cases} 0 \quad &\text{if $i \equiv 0 \Mod{r_j^{2k-1}}$},\\ 
                                1 \quad &\text{if $i\not \equiv 0 \Mod{r_j^{2k-1}}$}.\\
                                \end{cases}
$$                                
We also define
$$
\psi_{i,j}(k)=\begin{cases} \min\left(b_{2k-1}+\epsilon_{i,j}(k),\frac{r_j-1}2\right) \quad &\text{if $k<\left[\frac{e+1}2\right]$}, \\[0.6em]
                             \min\left(b_{2\delta-1}+\epsilon_{i,j}(\delta),\frac{r_j+1}2\right) \quad &\text{if $e=2\delta$ and $k=\delta$}, \\[0.6em]
                             1 \quad &\text{if $e=2\delta-1$ and $k=\delta$}. \\
                              \end{cases}
$$                              

\begin{lem} \label{upperbound} Under the notations and assumptions given above, we have 
$$
\Psi_{u,v}(i,j)\le  \displaystyle\sum_{k=1}^{\delta} \frac{r_j-1}2\left[\frac{i}{r_j^{2k}}\right]+\psi_{i,j}(k).
$$
\end{lem}
\begin{proof} Since both cases can be done in a similar manner, we only provide the proof of the case when $e=2\delta$ for some positive integer $\delta$. Without loss of generality, we may assume that $u=1$. We have to show that the number of integers of the form $r_{j}^{2k-1}\eta_{r_j}$  ($r_{j}^{2k-1}\eta_{r_j}'$) in the set $\{1+v,2+v,\dots,i+v\}$ is less than or equal to the right hand side, where $\eta_{r_j}$ ($\eta_{r_j}'$) is a square (nonsquare, respectively) in $\z_{r_j}^{\times}$.

For any integer $k$ such that $1\le k\le \delta$, let 
$$
i=r_{j}^{2k-1}(r_j{\alpha_k}+b_{2k-1})+\beta_k, \quad  (0\le \beta_k \le r_j^{2k-1}-1).
$$
Let $r_j^{2k-1}(x+1)$ be the smallest integer greater than $v$ that is divisible by $r_j^{2k-1}$. Then any integer in the set
$\{r_j^{2k-1}(x+s) : 1\le s\le r_j\alpha_k+b_{2k-1}\}$ is less than or equal to $i+v$. Note that there is at most one more integer other than these integers that is divisible by $r_j^{2k-1}$, and that is less than or equal to $i+v$. Note that such an integer exists only when $\epsilon_{i,j}(k)\ne 0$ (or $\beta_k \ne 0$). Furthermore, if such an integer exists, then it must be $r_j^{2k-1}(x+r_j\alpha_k+b_{2k-1}+1)$. 
Note that there are exactly $\frac{r_j-1}2$ quadratic residues and $\frac{r_j-1}2$ quadratic non-residues in the consecutive $r_j$ integers. 
Therefore there are exactly   $\frac{r_j-1}2\alpha_k$  quadratic residues and  $\frac{r_j-1}2\alpha_k$  quadratic non-residues in 
$$
\{r_j^{2k-1}(x+s) : 1\le s\le r_j\alpha_k\}.
$$ 
Note that $\alpha_k=\left[\frac{i}{r_j^{2k}}\right]$ for any $1\le k\le \delta$. The remaining multiples of $r_j^{2k-1}$ are  contained in
$$
\{r_j^{2k-1}(x+r_j\alpha_k+1), r_j^{2k-1}(x+r_j\alpha_k+2),\dots,  r_j^{2k-1}(x+r_j\alpha_k+b_{2k-1}+\epsilon_{i,j}(k))\}.
$$
Among them, there are at most $\psi_{i,j}(k)$ quadratic residues, and at most $\psi_{i,j}(k)$ quadratic non-residues. Note that there is at most one multiple of $r_j^{2\delta+1}$ in $\{1+v,2+v,\dots,i+v\}$ which is, if exists, contained in the set  
$$
\{r_j^{2\delta-1}(x+1), r_j^{2\delta-1}(x+2),\dots,  r_j^{2\delta-1}(x+b_{2\delta-1}+\epsilon_{i,j}(\delta))\}.
$$
 Note that there are at most $\psi_{i,j}(\delta)$ quadratic residues or a multiple of $r_j$, and at most  $\psi_{i,j}(\delta)$ quadratic non-residues or a multiple of $r_j$ in the set $\{x+1,x+2,\dots,x+b_{2\delta-1}+\epsilon_{i,j}(\delta)\}$.  The lemma follows from this. 
\end{proof}

For the sake of brevity, we let
$$
a_{ij}=\displaystyle\sum_{k=1}^{\delta} \frac{r_j-1}2\left[ \frac{i}{r_j^{2k}}\right] +\psi_{i,j}(k)
$$
for positive integers $i$ and $j$.

\begin{rmk} \label{simple}
One may easily show that $a_{ij}\le \left\lceil \displaystyle\frac{i}{r_j}\right\rceil$ for any positive integers $i$ and $j$, where $\left\lceil\cdot\right\rceil$ is the ceiling function.
It is a little bit complicate to compute an upper bound of $\Psi_{u,v}(i,j)$ by using Lemma \ref{upperbound}. Instead of that, one may easily show that
$$
\Psi_{u,v}(i,j) \le \displaystyle\frac{r_j+1}2\left\lceil \displaystyle\frac i{{r_j}^2}\right\rceil.
$$

\end{rmk}

Recall that $T$ is the set of all odd primes at which $\langle a,b,c\rangle$ is anisotropic, and $\vert T\vert =t\le 10$ by Lemma \ref{lemt10}.

\begin{lem} \label{lembound4}
Let $i$ be a positive integer. For any integer $s>t$, we define $b_{ij}(s)=\text{max}\left(a_{ij},\left\lceil \displaystyle\frac{i}{r_s} \right\rceil \right)$ for $j=1,2,\dots,s-1$. Then we have
$$
\vert \{ 1\le n\le i : 8n+a+b+c\overset{2}{\longrightarrow} \langle a,b,c\rangle \} \vert \ge i-b_{i,1}(s)-b_{i,2}(s)-\cdots -b_{i,s-1}(s).
$$
\end{lem}
\begin{proof}
Let $s$ be any integer greater than $t$ and let $J=\{ j\in \n : r_j\in T\}$. We also let $J_1=\{j\in J : j\le s-1\}$, $J_2=J - J_1$, and $J_3=\{1,2,\dots,s-1\}-J_1$. Note that $\vert J_2\vert \le \vert J_3\vert$ and for any $j \in J_3$, $\displaystyle \left\lceil \frac i{r_s}\right\rceil\le b_{ij}(s)$ by assumption.   
From Remark \ref{simple}, for any $j\in J_2$, we have $a_{ij}\le \left\lceil \displaystyle\frac{i}{r_j} \right\rceil \le \left\lceil \displaystyle\frac{i}{r_s} \right\rceil$.
Thus we have
$$
\begin{array}{rl}
\displaystyle\sum_{j\in J} a_{i,j}=\!\!\!\displaystyle\sum_{j_1\in J_1} a_{i,j_1}+\displaystyle\sum_{j_2\in J_2} a_{i,j_2} &\le \displaystyle\sum_{j_1\in J_1} a_{i,j_1}+\vert J_2 \vert \cdot \left\lceil \displaystyle\frac{i}{r_s} \right\rceil \\
&\le 
\displaystyle\sum_{j_1\in J_1} b_{i,j_1}(s)+\sum_{j_3\in J_3}b_{i,j_3}(s) \le
 \displaystyle\sum_{j=1}^{s-1} b_{i,j}(s).
\end{array}
$$
Since $\D(a,b,c)$ is stable regular,  we have
$$
\begin{array}{l}
\vert \{ 1\le n\le i : 8n+a+b+c\overset{2}{\longrightarrow} \langle a,b,c\rangle \} \vert \\[0.5em]
\hskip 5pc =\vert \{ 1\le n\le i : 8n+a+b+c\longrightarrow \gen(\langle a,b,c\rangle) \} \vert \\[0.5em]
\hskip 5pc \ge i-\displaystyle\sum_{j\in J} a_{i,j} \ge i-\displaystyle\sum_{j=1}^{s-1} b_{i,j}(s).
\end{array}
$$
This completes the proof.
\end{proof}

\begin{rmk} In the remaining of this section, we need the exact values of $a_{ij}$'s for some integers $i$ and $j$.  We provide some of these values in Table 1 below.
\end{rmk}

\begin{table}[ht]
\caption{Some values of $a_{ij}$}
\begin{tabular}{|c|c|c|c|c|c|c|c|c|c|c|c|}
\hline
\diagbox[width=3em,height=1.7em]{$i$}{$j$}
&1&2&3&4&5&6&7&8&9&10&11 \\
\hline
1&   1&1&1&1&1&1&1&1&1&1&1 \\
\hline
2&   1&1&1&1&1&1&1&1&1&1&1 \\
\hline
3&   1&1&1&1&1&1&1&1&1&1&1 \\
\hline
4&   2&1&1&1&1&1&1&1&1&1&1 \\
\hline
5&   2&1&1&1&1&1&1&1&1&1&1 \\
\hline
6&   2&2&1&1&1&1&1&1&1&1&1 \\
\hline
7&   2&2&1&1&1&1&1&1&1&1&1 \\
\hline
9&   2&2&2&1&1&1&1&1&1&1&1 \\
\hline
19&  4&3&3&2&2&2&1&1&1&1&1 \\
\hline
20&  4&3&3&2&2&2&2&1&1&1&1 \\
\hline
25&  4&3&4&3&2&2&2&2&1&1&1 \\
\hline
26&  4&4&4&3&2&2&2&2&1&1&1 \\
\hline
29&  6&4&4&3&3&2&2&2&1&1&1 \\
\hline
32&  6&5&4&3&3&2&2&2&2&2&1 \\
\hline
35&  6&5&4&4&3&3&2&2&2&2&1 \\
\hline
41&  7&5&4&4&4&3&3&2&2&2&2 \\
\hline
47&  8&5&4&5&4&3&3&3&2&2&2 \\
\hline
49&  8&5&4&5&4&3&3&3&2&2&2 \\
\hline
83& 13&9&7&6&7&5&5&4&3&3&3 \\
\hline
314& 41&29&22&16&13&11&10&12&11&11&9 \\
\hline
\end{tabular}
\end{table}

\begin{lem} \label{lemt7}
Under the assumptions given above, we have $t\le 7$.
\end{lem}

\begin{proof}
By Lemma \ref{lembound4} with $i=25$ and $s=11$, one may easily show that 
$8n_1+a+b+c\overset{2}{\longrightarrow}\langle a,b,c\rangle$ for some 
$1\le n_1\le 25$. 
From our assumption of $a\le b\le c$, we have $9a+b+c\le 8n_1+a+b+c$, and thus we have
$a\le 25$. To prove the lemma, we will use Lemma \ref{keylem} repeatedly.  

First, assume that $q_1\ge 7$. Since $t'-1<11\le q_2$, we may apply  Lemma \ref{keylem} so that
$$
q_1q_2\cdots q_{t'}<25\left(3q_1^2t'\right)^2.
$$
This is possible only when $t'\le 6$.  Now, assume that $q_1=5$.  Since $t'-2<11\le q_3$, 
one may deduce that 
$$
q_1q_2\cdots q_{t'}<25\left(3q_1^2q_2(t'-1)\right)^2
$$
and thus $t'\le 7$. Finally, since $t'-1<7\le q_2$, we have
$$
q_1q_2\cdots q_{t'}<25(3q_1^2 t')^2
$$
and thus $t'\le 6$.
This completes the proof.
\end{proof}

\begin{lem} \label{a=2} For any stable regular ternary triangular form $\Delta(a,b,c)$ with $0<a\le b\le c$, we have $a=1$ or $2$.   
\end{lem}

\begin{proof}
For any positive integer $n$, we define $s_n=8n+a+b+c$. 
Since
$$
\begin{array}{l}
\{ s_n : s_n<25a+b+c,\ s_n\overset{2}{\longrightarrow} \langle a,b,c\rangle \} \\[0.4em]
\subset \{ 9a+b+c,\ a+9b+c,\ a+b+9c,\ 9a+9b+c,\ 9a+b+9c,\ a+9b+9c\},
\end{array}
$$
we have
$$
\vert \{ 1\le n\le 3a-1 : s_n\overset{2}{\longrightarrow} \langle a,b,c\rangle \} \vert \le 6.
$$
On the other hand, by Lemma \ref{lembound4} with $i=32$ and $s=8$, one may check that
$$
\vert \{ 1\le n\le 32 : s_n\overset{2}{\longrightarrow} \langle a,b,c\rangle \} \vert \ge 7.
$$
By comparing these two inequalities, we have $a\le 10$.

Now, we will show that if $3\le a \le 10$, then $c$ is bounded. 
For each positive odd integer $k$, we let
$$
\begin{array}{l}
U_k(a,b,c)=\left\{ 1\le n< \displaystyle\frac{k^2-1}8 a : s_n\overset{2}{\longrightarrow} \langle a,b,c\rangle \right\}, \\[1.1em]
V_k(a,b,c)=\left\{ 1\le n< \displaystyle\frac{k^2-1}8 a : s_n-c \overset{2}{\longrightarrow} \langle a,b\rangle \right\},
\end{array}
$$
and we also let $u_k=\vert U_k\vert$ and $v_k=\vert V_k\vert$.
Note that $V_k$ does not depend on $c$. For each integer $a$ with $3\le a\le10$, we will choose an integer $k$ suitably so that $v_k < u_k$. Note that if this inequality holds, then 
$a+b+9c \le 8(\frac{k^2-1}8a-1)+a+b+c$ and therefore, we have
$$
c\le \frac{k^2-1}8a-1.
$$  
In fact, we choose
$$
(a,k)=(10,5),(9,5),(8,7),(7,7),(6,7),(5,9),(4,13)\ \text{and}\ (3,29).
$$
Now, by using Lemma \ref{lembound4} with $i=\displaystyle\frac{k^2-1}8 a-1$ and $s=8$, one may easily compute the lower bound of $u_k$:   
\begin{table} [ht]
\begin{tabular}{|c|c|c|c|c|c|c|c|c|}
\hline
$(a,k)$ & (10,5) & (9,5) & (8,7) & (7,7) & (6,7) & (5,9) & (4,13) & (3,29) \\
\hline
$u_k$ & $\ge 5$ & $\ge 5$ & $\ge 15$ & $\ge 11$ & $\ge 8$ & $\ge 17$ &
 $\ge 31$ & $\ge 164$ \\
\hline
\end{tabular}
\end{table}

\noindent To compute an upper bound of $v_k$, note that
$$
V_k=\{ \alpha^2a+\beta^2b : a+b<\alpha^2a+\beta^2b<k^2a+b,\ \alpha \beta\equiv 1\Mod 2 \}.
$$
 Hence one may easily show that
$$
v_5\le 3,\ v_7\le 7,\ v_9\le 14,\ v_{13}\le 30 \quad \text{and}\quad v_{29}\le 161.
$$
By comparing the lower bound for $u_k$ and the upper bound for $v_k$, we have an upper bound of $c$ for each $a=3,4,\cdots,10$, as follows:

\begin{table} [ht]
\begin{tabular}{|c|c|c|c|c|c|c|c|c|}
\hline
$a$&10&9&8&7&6&5&4&3 \\
\hline
$c$ & $\le 29$ & $\le 26$ & $\le 47$ & $\le 41$ & $\le 35$ & $\le 49$ & $\le 83$ & $\le 314$ \\
\hline
\end{tabular}
\end{table}

\noindent Now, by using MAPLE program, one may check that there is no stable regular ternary triangular form $\D(a,b,c)$ for $3\le a\le 10$. Therefore, we have $a\le 2$.
\end{proof}

\begin{lem} \label{tle5}
Under the assumptions given above, we have $t\le 5$.
\end{lem}
\begin{proof}
By the proof of Lemma \ref{lemt7}, we have $t'\le 6$. First, assume that $a=2$. By Lemma \ref{lembound4} with $i=29$ and $s=8$, one may easily show, by using Table 1, that
$$
\vert \{ 1\le n\le 29 : s_n\overset{2}{\longrightarrow} \langle 2,b,c\rangle \} \vert \ge 5.
$$
On the other hand,
$$
\begin{array}{l}
\vert \{ 1\le n\le 29 : 8n+2+b+c=2\alpha^2+b+c\ \text{for some odd integer}\ \alpha \} \vert \\[0.4em]
\hskip 3pc =\vert \{\alpha\ge 3 : 2\alpha^2+b+c \le 8\cdot29+2+b+c,\ \alpha \equiv 1\Mod 2\} \vert=4. \\[0.4em]
\end{array}
$$
Thus we have $2+9b+c\le 8\cdot29+2+b+c$ and $b\le 29$.
Let  $g$ be a positive integer satisfying Lemma \ref{lembound3} in the case when $p=q_1$ and $d=24$. Note that
$$
24q_1^2n+24g+2+b+c\longrightarrow \langle 2,b,c\rangle \ \ \text{over}\ \ \z_3
$$
for any integer $n$. For any positive integer $r$, define 
$$
h(r)=24q_1^2(r-1)+24g+2+b+c.
$$ 
Clearly $h(r)$ is represented by $\langle 2,b,c\rangle$ over $\z_q$ for any $q \in \{2,3,q_1\}$.
Note that 
$$
t'-1 \le 5,\ \ b_{7,2}(6)=2\quad \text{and}\quad b_{7,j}(6)=1\ \ \text{for any}\ \ j\ge 3,
$$
where $b_{ij}(s)$ is an integer defined in Lemma \ref{lembound4}.
From this, similarly with the proof of Lemma \ref{lembound4}, one may easily show that there exists a positive integer $r$ with $1\le r\le7$ such that $h(r)$ is represented by  $\langle 2,b,c\rangle$ over $\z_{q_{i}}$ for any $i=2,3,\dots,t'$.
Therefore, we have 
$$
h(r)=24q_1^2(r-1)+24g+2+b+c\longrightarrow \gen(\langle 2,b,c\rangle).
$$
Furthermore, since $\D(2,b,c)$ is regular, we have 
$$
h(r)=24q_1^2(r-1)+24g+2+b+c\overset{2}{\longrightarrow}\langle 2,b,c\rangle.
$$
From our choices of $g$ and $r$, we have $h(r)-c\nrightarrow \langle 2,b\rangle$. 
Thus, $2+b+9c\le h(r)$, which implies that $c \le 21q_1^2$. Therefore we have 
$$
q_1q_2\cdots q_{t'} \le abc \le 58c \le 1218q_1^2.
$$ 
This implies that $t'\le 4$.

Now, assume that $a=1$. By Lemma \ref{lembound4} with $i=35$ and $s=8$, one may easily show that
$$
\vert \{ 1\le n\le 35 : s_n\overset{2}{\longrightarrow}\langle 1,b,c\rangle \} \vert \ge 8.
$$
On the other hand,
$$
\begin{array}{l}
\vert \{ 1\le n\le 35 : 8n+1+b+c=\alpha^2+b+c\ \ \text{for some odd integer}\ \ \alpha \} \vert \\[0.4em]
\hskip 3pc=\vert \{ \alpha\ge 3 : \alpha^2+b+c \le 8\cdot 35+1+b+c,\ \alpha \equiv 1\Mod 2 \} \vert=7. \\[0.4em] 
\end{array}
$$
Thus we have $1+9b+c \le 8\cdot 35+1+b+c$ and $b\le 35$. 
Similarly to the case when $a=2$, one may deduce that $c \le 21q_1^2$. Therefore, we have
$$
q_1q_2\cdots q_{t'} \le abc \le 35c \le 735q_1^2,
$$ 
which implies that $t'\le 4$. This completes the proof.
\end{proof}

\begin{lem} \label{upperb}  For any stable regular ternary triangular form $\Delta(a,b,c)$ with $0<a\le b\le c$, we have $a+b \le21$.   
\end{lem}

\begin{proof}
Note that $a=1$ or $2$ by Lemma \ref{a=2}.
First, assume that $a=2$. By Lemma \ref{lembound4} with $i=19$ and $s=6$, one may easily show that
$$
\vert \{ 1\le n\le 19 : 8n+2+b+c\overset{2}{\longrightarrow}\langle 2,b,c\rangle \} \vert \ge 5.
$$
On the other hand,
$$
\begin{array}{l}
\vert \{ 1\le n\le 19 : 8n+2+b+c=2\alpha^2+b+c\ \ \text{for some odd integer}\ \ \alpha \} \vert \\[0.4em]
\hskip 3pc =\vert \{ \alpha\ge3 : 2\alpha^2+b+c \le 8\cdot19+2+b+c,\ \alpha \equiv 1\Mod 2\} \vert=3. \\[0.4em]
\end{array}
$$
Thus we have $2+9b+c \le 8\cdot 19+2+b+c$, and $b\le 19$. Now, assume that $a=1$. By Lemma \ref{lembound4} with $i=20$ and $s=6$, one may check that
$$
\vert \{ 1\le n\le 20 : 8n+1+b+c\overset{2}{\longrightarrow}\langle 1,b,c\rangle \} \vert \ge 6.
$$
On the other hand,
$$
\begin{array}{l}
\vert \{ 1\le n\le 20 : 8n+1+b+c=\alpha^2+b+c\ \ \text{for some odd integer}\ \ \alpha \} \vert \\[0.4em]
\hskip 3pc =\vert \{ \alpha\ge3 : \alpha^2+b+c \le 8\cdot20+1+b+c,\ \alpha \equiv 1\Mod 2\} \vert=5. \\[0.4em]
\end{array}
$$
Thus we have $1+9b+c \le 8\cdot20+1+b+c$, and $b\le 20$. 
\end{proof}

Now, we are ready to classify all stable regular ternary triangular forms. The following lemma is very useful to prove the regularity.

\begin{lem} \label{lem13}
Let $m$ be a positive integer congruent to 4 modulo 8. Then
$$
r_{(1,1)}(m,\langle 1,3\rangle)=\frac23 r(m,\langle 1,3\rangle).
$$
\end{lem}
\begin{proof}
See \cite[Lemma 3.1(iii)]{KO}.
\end{proof}

\begin{thm} \label{thmfloor}
There are exactly $17$ stable regular ternary triangular forms.
$$
\begin{array}{llll} 
\Delta_1=\Delta(1,1,1),&\Delta_2=\Delta(1,1,2),&\Delta_3=\Delta(1,1,3),&
\Delta_4=\Delta(1,1,4),\\[0.2em]
\Delta_5=\Delta(1,2,2),&\Delta_6=\Delta(1,1,5),&\Delta_7=\Delta(1,1,6),&
\Delta_8=\Delta(1,2,3),\\[0.2em]
\Delta_9=\Delta(1,2,4),&\Delta_{10}=\Delta(1,2,5),&\Delta_{11}=\Delta(1,1,12),&
\Delta_{12}=\Delta(1,3,4),\\[0.2em]
\Delta_{13}=\Delta(2,2,3),&\Delta_{14}=\Delta(1,2,10),&
\Delta_{15}=\Delta(1,1,21),&\Delta_{16}=\Delta(1,4,6),\\[0.2em]
\Delta_{17}=\Delta(1,3,10).&&&
\end{array}
$$
\end{thm}
\begin{proof}
By Lemmas \ref{a=2}, \ref{tle5} and \ref{upperb},  we have
$$
t \le 5, \ \ 1\le a\le2, \ \ \text{and} \ \ a+b\le 21.
$$
First, we want to find an upper bound for $c$ for each possible pair $(a,b)$. 
Since all the other cases can be done in a similar manner, we only consider 3 representative cases here.

{\bf \textrm{(i)} $(a,b)=(2,2)$}. 
Let $E_1=\{ 4\!\cdot\!3, 4\!\cdot\!7, 4\!\cdot\!11, 4\!\cdot\!19, 4\!\cdot\!23, 4\!\cdot\!31\}$. 
Suppose that $c\ge 16$. For any $e_1\in E_1$, $e_1$ is not represented by $\langle 2,2\rangle$. Furthermore, since $e_1+c <4+9c$ by assumption,
 $e_1+c\overset{2}{\nrightarrow}\langle 2,2,c\rangle$.   
Since $\Delta(2,2,c)$ is stable regular, there is an odd prime divisor $p$ of $e_1+c$ such that $\langle 2,2,c\rangle$ is anisotropic over $\z_{p}$. Therefore, $p$ divides $c$ and also divides $e_1$.
Furthermore, since $\vert E_1\vert =6$, there are at least six such odd primes. 
This is a contradiction to the fact that $t\le 5$. 
Thus, we have $c\le 15$ if $(a,b)=(2,2)$.

{\bf \textrm{(ii)} $(a,b)=(2,3)$}.
Let $E_2=\{ 69,117,141,213,285,333\}$. Suppose that $c\ge 42$. 
 Since we are assuming that $\Delta(2,3,c)$ is 3-stable, $c$ is not divisible by 3. 
Any element of $E_2$ is of the from 
$8n+2+3$ for some positive integer $n$, and the elements of $E_2$ share no odd prime divisors other than $3$.
Let $e_2 \in E_2$. From the assumption that $c\ge 42$, one may easily check that $e_2+c\overset{2}{\nrightarrow} \langle 2,3,c\rangle$.  Since $\D(2,3,c)$ is stable regular, 
there is an odd prime $p$ dividing $e_2+c$ and $\langle 2,3,c\rangle $ is anisotropic over $\z_p$. Hence $p$ is greater than $3$ and divides $e_2$.  
Thus there are at least six such odd primes. This is a contradiction, and we have $c\le 41$. 

{\bf \textrm{(iii)} $(a,b)=(2,6)$}. 
Since $\D(2,6,c)$ is 3-stable, $c$ is not a multiple of $3$. 
Note that $48+c=8\cdot 5+2+6+c \overset{2}{\nrightarrow} \langle 2,6,c\rangle$. 
Thus there is an odd prime $p>3$ dividing $48+c$ and 
$\langle 2,6,c\rangle$ is anisotropic over $\z_{p}$. Therefore, $48$ is divisible by $p$, which is a contradiction. 
Therefore,  the pair $(a,b)=(2,6)$ is impossible.

All the other cases can be done in a similar manner to one of the above three cases, and one may obtain an upper bound for $c$ in each case.  After that, with the help of MAPLE program, one may show that there are $17$  candidates of stable regular ternary triangular forms given above. 

For each $i=1,2,\cdots,17$, we write $\Delta_i=\Delta(a_i,b_i,c_i)$ and 
$L_i=\langle a_i,b_i,c_i\rangle$.
For any $i \in U=\{1,2,4,5,6,8,9\}$, it is well known that $\Delta_i$ is universal (see \cite[p.23]{D99}). 
Hence we may assume that $i\not\in U$.
Let $n_i$ be any positive integer such that 
$$
\widetilde{n_i}:=8n_i+a_i+b_i+c_i \longrightarrow \gen(L_i).
$$
Note that $L_i$ has class number 1 for any $1\le i\le 17$ and thus $\widetilde{n_i}\longrightarrow L_i$.

For $i\in \{ 11,13,14,15,16\}$, one may easily check that
$$
R(\widetilde{n_i},L_i)=R_{(1,1,1)}(\widetilde{n_i},L_i),
$$
that is, if $a_ix^2+b_iy^2+c_iz^2=\widetilde{n_i}$, then $xyz \equiv 1 \Mod 2$. 
Assume that $i\in \{ 7,10\}$. 
Since the class number of $L_i$ is 1 and it primitively 
represents $\widetilde{n_i}$ over $\z_2$, there is a vector 
$(x,y,z)\in R(\widetilde{n_i},L_i)$ with $(x,y,z,2)=1$. One may easily check 
that $(x,y,z,2)=1$ implies $xyz\equiv 1\Mod 2$ in this case. 
If $i=12$, 
then one may easily show that
$$
r(\widetilde{n_i},L_i)=r_{(0,0,0)}(\widetilde{n_i},L_i)+r_{(0,0,1)}(\widetilde{n_i},L_i)+r_{(1,1,1)}(\widetilde{n_i},L_i).
$$
Similarly to the previous case, the existence of a vector 
$(x,y,z)\in R(\widetilde{n_i},L_i)$ with $(x,y,z,2)=1$ implies that
$$
r_{(0,0,1)}(\widetilde{n_i},L_i)+r_{(1,1,1)}(\widetilde{n_i},L_i)>0.
$$
By Lemma \ref{lem13},
$$
\begin{array}{rl}
r_{(1,1,1)}(8n_i+8,x^2+3y^2+4z^2)&=\displaystyle\sum_{z:\text{odd}}r_{(1,1)}(8n_i+8-4z^2,x^2+3y^2) \\
&=\displaystyle\sum_{z:\text{odd}}\frac23 r(8n_i+8-4z^2,x^2+3y^2) \\ [1.6em]
&=\displaystyle\frac23 r_{(0,0,1)}(\widetilde{n_i},L_i)+\displaystyle\frac 23r_{(1,1,1)}(\widetilde{n_i},L_i).
\end{array}
$$
Therefore we have $r_{(1,1,1)}(\widetilde{n_i},x^2+3y^2+4z^2)=2r_{(0,0,1)}(\widetilde{n_i},x^2+3y^2+4z^2)>0$.
If $i=3$, then one may easily check that
$$
r(\widetilde{n_i},L_i)=2r_{(1,0,0)}(\widetilde{n_i},L_i)+r_{(1,1,1)}(\widetilde{n_i},L_i).
$$
By Lemma \ref{lem13}, we have
$$
\begin{array}{rl}
r_{(1,1,1)}(8n_i+5,x^2+y^2+3z^2)&=\displaystyle\sum_{x:\text{odd}}r_{(1,1)}(8n_i+5-x^2,y^2+3z^2) \\[1.6em]
&=\displaystyle\sum_{x:\text{odd}} 2r_{(0,0)}(8n_i+5-x^2,y^2+3z^2) \\[1.6em]
&=2r_{(1,0,0)}(8n_i+5,x^2+y^2+3z^2).
\end{array}
$$
Thus we have $r_{(1,1,1)}(\widetilde{n_i},x^2+y^2+3z^2)=\frac12 r(\widetilde{n_i},x^2+y^2+3z^2)>0$.
Finally, assume that $i=17$. 
Note that if $x^2+3y^2+10z^2=8n+14$, then $x\equiv y\Mod 2$ and $z\equiv 1\Mod 2$. 
By Lemma \ref{lem13} again, we have
$$
\begin{array}{rl}
r_{(1,1,1)}(8n_i+14,x^2+3y^2+10z^2)&=\displaystyle\sum_{z\in \z}r_{(1,1)}(8n_i+14-10z^2,x^2+3y^2) \\[1.6em]
&=\displaystyle\sum_{z\in \z}\frac23 r(8n_i+14-10z^2,x^2+3y^2) \\[1.6em]
&=\displaystyle\frac23 r(8n_i+14,x^2+3y^2+10z^2).
\end{array}
$$
This completes the proof.
\end{proof}

\section{Regular ternary triangular forms}

In this section, we prove that there are exactly 49 regular ternary triangular forms.
Let $\D(a',b',c')$ be a regular  ternary  triangular form  and let $\D(a,b,c)$ be the stable regular ternary triangular form obtained from it by taking $\lambda$-transformations, if necessary, repeatedly.
Here, we are not assuming that $a\le b \le c$.
It might happen that there is an odd prime $l$ dividing $a'b'c'$ such that $(abc,l)=1$.  
We call such a prime $l$ a \textit{missing prime.}
Note that $\lambda_p \circ \lambda_q=\lambda_q \circ \lambda_p$ for any odd primes $p$ and $q$. 
Thus if $l$ is a missing prime, then one of the followings holds:
\begin{enumerate} [(i)]
\item $\Delta(a,l^2b,l^2c)$ is regular.
\item $\Delta(a,b,l^2c)$ is regular and $\left( \displaystyle\frac{-ab}{l} \right)=-1$.
\end{enumerate}

\begin{lem} \label{lemmp}
There is no missing prime $l$ greater than $7$.
\end{lem}

\begin{proof}
Let $l$ be a missing prime. Then there is a stable regular ternary triangular form $\Delta(a,b,c)$ such that $(abc,l)=1$, and
(i) or (ii) given above holds. 

Assume that the case (i) holds, that is,  $\Delta(a,l^2b,l^2c)$ is regular. We let
$$
s_n=8n+a+l^2b+l^2c\quad \text{for}\ \ n=1,2,3,\cdots.
$$
First, we prove that $l\le 131$. Assume to the contrary that $l\ge 137$. One may easily check that if 
$$
\alpha^2 a+\beta^2 l^2 b+\gamma^2 l^2 c \le 8l+a+l^2 b+l^2 c
$$ 
with odd integers $\alpha, \beta$ and $\gamma$, then $\beta^2 =\gamma^2 =1$. 
Thus we have
$$
\left\vert \left\{ 1\le n\le l : s_n\overset{2}{\longrightarrow} 
\langle a,l^2b,l^2c \rangle \right\} \right\vert
\le \left[ \sqrt{\displaystyle\frac{2l}{a}+\frac14}-\frac12 \right]
\le \left[ \sqrt{2l+\displaystyle\frac14} \right].
$$
On the other hand, by Theorem \ref{thmfloor}, 
the set of odd primes at which $\langle a,b,c \rangle$ is anisotropic is
$$
\emptyset ,\ \{ 3 \} ,\ \{ 5 \} ,\ \{ 7 \} ,\ \{ 3,5 \} \ \ \text{or}\ \ \{ 3,7 \} .
$$
From Remark \ref{simple}, we have
$$
\vert \{ 1\le n\le l : s_n\nrightarrow 
\langle a,l^2 b,l^2 c \rangle \text{ over } \z_p \} \vert \le 
\begin{cases} 
2\left\lceil \displaystyle\frac{l}{9} \right\rceil&\text{if}\ \ p=3, \\[10pt]
3\left\lceil \displaystyle\frac{l}{25} \right\rceil&\text{if}\ \ p=5, \\[10pt]
4\left\lceil \displaystyle\frac{l}{49} \right\rceil&\text{if}\ \ p=7, \\[10pt]
\displaystyle\frac{l+1}{2}&\text{if}\ \ p=l.
\end{cases}
$$
From the assumption that $l\ge 137$, we have $\displaystyle\frac{3}{25}l+3 \ge \frac{4}{49}l+4$. 
Since 
$$
l-\left( \displaystyle\frac{2}{9}l+2+\frac{3}{25}l+3+\frac{l+1}{2} \right) 
=\frac{71}{450}l-\frac{11}{2},
$$
we must have
$$
\left\vert \left\{ 1\le n\le l : s_n\overset{2}{\longrightarrow} 
\langle a,l^2b,l^2c \rangle \right\} \right\vert
\ge \left\lceil \displaystyle\frac{71}{450}l-\frac{11}{2} \right\rceil .
$$
However, one may directly  show that if $l\ge 137$, then 
$\left\lceil \displaystyle\frac{71}{450}l-\frac{11}{2} \right\rceil 
> \left[ \sqrt{2l+\frac14} \right]$. 
This is a contradiction and hence we have $l\le 131$. 
Now, by a direct calculation with the help of MAPLE, 
one may check that for any prime $11\le q\le 131$ and any stable regular ternary triangular form $\Delta(a,b,c)$, 
all of the triangular forms $\Delta(a,q^2 b,q^2 c)$ are not regular.

Now, assume that $\Delta(a,b,l^2 c) \ (a\le b)$ is regular with 
$\left( \displaystyle\frac{-ab}{l} \right) =-1$.
 By Theorem \ref{thmfloor}, $(a,b)$ is one of the following pairs:
$$
\begin{array}{lllllllll}
(1,1),&(1,2),&(1,3),&(1,4),&(2,2),&(1,5),&(1,6),&(2,3),&(2,4),\\
(1,10),&(2,5),&(1,12),&(3,4),&(2,10),&(1,21),&(4,6),&(3,10).& 
\end{array}
$$
 First, suppose that $l\ge 29$.
Since all the other cases can be done in a similar manner, we only consider the cases when $(a,b)=(1,1)$ or $(1,5)$.
Assume that $(a,b)=(1,1)$. Since
$$
418+l^2 c=8\cdot 52+1+1+l^2 c \longrightarrow \gen(\langle 1,1,l^2c\rangle),
$$
and $\Delta(a,b,l^2 c)$ is regular, there is a vector $(x,y,z)\in z^3$ with 
$xyz\equiv 1 \Mod 2$ such that $x^2+y^2+l^2 cz^2=418+l^2 c$. 
From the assumption that $l\ge 29$, we have $z^2=1$. 
This is a contradiction, for $418$ is not a sum of two integer squares. 
Next, assume that $(a,b)=(1,5)$. Note that
$$
110+l^2 c=8\cdot 13+1+5+l^2 c\longrightarrow \gen(\langle 1,5,l^2c\rangle).
$$
 Since we are assuming that $\D(1,5,l^2c)$ is regular, there is a vector 
$(x_1,y_1,z_1)\in \z^3$ with $x_1y_1z_1\equiv 1\Mod 2$ such that 
$x_1^2+5y_1^2+l^2 cz_1^2=110+l^2 c$. 
Since $l\ge 29$, we have $z_1^2=1$. This is a contradiction, for $110$ is not represented by $\langle 1,5 \rangle$.
Therefore, we have $l\le 23$. Now, by  a direct calculation with the help of MAPLE,
one may check that for any prime $11\le l\le 23$ and any stable regular ternary triangular form $\Delta(a,b,c)$,
all of the forms $\Delta(a,b,l^2 c)$ are not regular. 
This completes the proof.
\end{proof}

\begin{rmk}
By Theorem \ref{thmfloor} and Lemma \ref{lemmp}, any prime divisor of the discriminant of a regular ternary triangular form is less than or equal to $7$.
\end{rmk}

Let $\Delta(a',b',c')$ be a regular ternary triangular form. Then there are nonnegative integers $e_3,e_5$ and $e_7$ such that 
$$
\lambda_3^{e_3}(\lambda_5^{e_5}(\lambda_7^{e_7}(\Delta(a',b',c'))))=\Delta(a,b,c),
$$
is stable regular. Hence, to find all regular ternary triangular forms, it suffices to find all regular ternary triangular forms in the inverse image of the $\lambda_p$-transformation of each regular triangular form for each $p \in \{3,5,7\}$. Note that any triangular form in the inverse image $\lambda_p^{-1}(\Delta(a,p^rb,p^sc))$, for $abc \not \equiv 0 \Mod p$ and $0\le r \le s$, is given in Table 2.

\begin{table} [ht] 
\caption{Inverse image of $\lambda_p$-transformations}
{\begin{tabular}{|c|c|}
\hline
Cases & Triangular forms in $\lambda_p^{-1}(\D(a,p^rb,p^sc))$ \\
\hline
$r=s=0$ & $ \begin{array}{c} \D(p^2a,b,c),\D(a,p^2b,c),\D(a,b,p^2c), \\[0.1em] 
\D(p^2a,p^2b,c),\D(p^2a,b,p^2c),\D(a,p^2b,p^2c) \end{array}$ \\
\hline
$r=0,s=1$ & $\Delta(pa,pb,c),\Delta(a,p^2b,p^3c),\Delta(p^2a,b,p^3c),\Delta(a,b,p^3c)$ \\
\hline
$r=0,s\ge 2$ & $\Delta(a,p^2b,p^{s+2}c),\Delta(p^2a,b,p^{s+2}c),\Delta(a,b,p^{s+2}c)$ \\
\hline
$r=s=1$ & $\Delta(pa,b,p^2c),\Delta(pa,p^2b,c),\Delta(pa,b,c),\Delta(a,p^3b,p^3c)$ \\
\hline
$r=1,s\ge 2$ & $\Delta(pa,b,p^{s+1}c),\Delta(a,p^3b,p^{s+2}c)$ \\
\hline
$r\ge 2$ & $\Delta(a,p^{r+2}b,p^{s+2})$ \\
\hline
\end{tabular}}
\end{table}

First, we find all regular triangular forms in the inverse images of stable regular ternary triangular forms via $\lambda_p$-transformation for each $p \in \{3,5,7\}$, and then we repeat this process again until any inverse image does not contain a regular triangular form. 
As a sample, ternary triangular forms lying over $\Delta(1,1,1)$ are given in Table 3. In that table, if the triangular form is not regular, then the smallest positive integer which is represented locally, but not globally by the triangular form is given in parentheses.

\begin{table}
\caption{Triangular forms lying over $\Delta(1,1,1)$ via $\lambda$-transformations}
{\begin{tikzpicture}
      [
      arrow/.style = {draw, -latex},
      ]

      \coordinate (A3) at (0cm, 5.5cm);
      \coordinate (A2) at (0cm, 5cm);
      \coordinate (A1) at (0cm, 3.5cm);

      \coordinate (B2) at (3cm, 5.25cm);
      \coordinate (B1) at (3cm, 3.5cm);

      \coordinate (C2) at (6cm, 3.75cm);
      \coordinate (C1) at (6cm, 3.25cm);

      \coordinate (D2) at (9cm, 3.75cm);
      \coordinate (D1) at (9cm, 3.25cm);

      \coordinate (O) at (4.5cm, 0.5cm);

      \node at (A3) {$\D(1,1,81)(19)$};
      \node at (A2) {$\D(1,9,81)(19)$};
      \node at (A1) {$\D(1,1,9)$};

      \node at (B2) {$\D(1,81,81)(19)$};
      \node at (B1) {$\D(1,9,9)$};

      \node at (C2) {$\D(1,1,25)(5)\;\;$};
      \node at (C1) {$\D(1,25,25)(5)$};

      \node at (D2) {$\D(1,1,49)(8)\;\;$};
      \node at (D1) {$\D(1,49,49)(8)$};

      \node at (O) {$\D(1,1,1)$};

      \path [arrow] ([yshift=-3mm]A2) -- node [below][xshift=3mm,yshift=3mm] {$\lambda_{3}$} ([yshift=3mm]A1);
      \path [arrow] ([yshift=-3mm]B2) -- node [below][xshift=3mm,yshift=3mm] {$\lambda_{3}$} ([yshift=3mm]B1);

      \path [arrow] ([yshift=-3mm]A1) -- node [below][xshift=-1.5mm,yshift=0.5mm] {$\lambda_{3}$} ([xshift=-6mm,yshift=3mm]O);
      \path [arrow] ([yshift=-3mm]B1) -- node [below][xshift=-1.8mm,yshift=2.5mm] {$\lambda_{3}$} ([xshift=-2mm,yshift=3mm]O);
      \path [arrow] ([yshift=-3mm]C1) -- node [below][xshift=3mm,yshift=3mm] {$\lambda_{5}$} ([xshift=2mm,yshift=3mm]O);
      \path [arrow] ([yshift=-3mm]D1) -- node [below][xshift=2mm,yshift=0.8mm] {$\lambda_{7}$} ([xshift=6mm,yshift=3mm]O);
\end{tikzpicture}}
\end{table}

Finally, one may have a list of $49$ candidates for the regular ternary triangular forms including 17 stable regular forms,  which is given in Table 4.  The regularities of $32$ forms except $17$ stable regular forms will be proved here. Before doing that, we need some lemmas. 

Let $p$ be an odd prime and let $k$ be a positive integer relatively prime to $p$. 
Assume that $p$ is represented by the binary quadratic form $x^2+ky^2$. 
In 1928, B. W. Jones proved in his unpublished thesis that if the Diophantine equation $x^2+ky^2=N (N>0)$ has an integral solution, then it also has an integral solution $x,y$ with $(x,y,p)=1$.
The following lemma follows immediately from this.

\begin{lem} \label{lem12}
Let $N$ be a positive integer. If $x^2+2y^2=N$ for some $(x,y)\in \z^2$, 
then there is a vector $(\tilde{x},\tilde{y})\in \z^2$ such that
$$
\tilde{x}\not\equiv \tilde{y}\Mod 3,\ \tilde{x}\equiv x\Mod 4,\ \tilde{y}\equiv y\Mod 2 \ \ \text{and} \ \ \tilde{x}^2+2\tilde{y}^2=N.
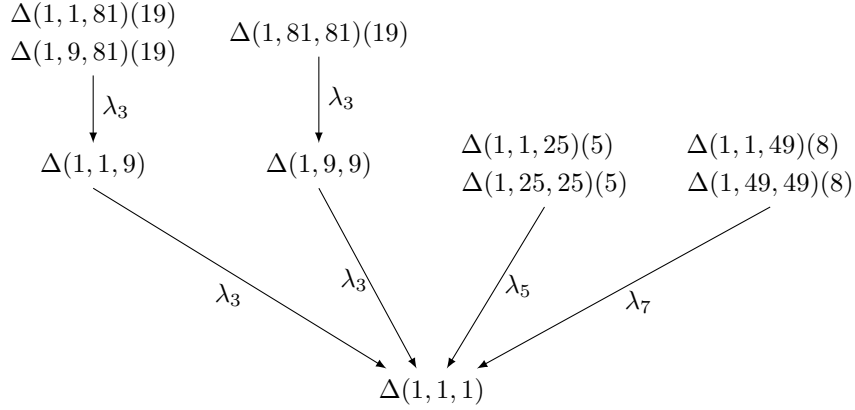
$$
\end{lem}

We also need the following lemma which appeared in the middle of the proof of \cite[Theorem 3.1]{O1}.

\begin{lem} \label{lemeigen}
Let $S\in M_3(\z)$ be a positive-definite symmetric matrix and let $T\in M_3(\q)$ such that ${}^tTST=S$. Let $(u,v,w)\in \z^3$ and define
$$
\begin{pmatrix} u_n\\v_n\\w_n \end{pmatrix}=T^n\begin{pmatrix} u\\v\\w\end{pmatrix},\quad n=1,2,3,\cdots.
$$
Assume that 
\begin{enumerate} [(i)]
\item $T$ has an infinite order.
\item $(u_n,v_n,w_n)\in \z^3$ for any $n$.
\end{enumerate}
Then $(u,v,w)\in {\rm ker} \left( T-{\rm det} (T)I\right)$ and ${\rm dim}_{\mathbb{R}}{\rm ker}\left((T-{\rm det}(T)I)\right)=1$.
\end{lem}

In the following 5 consecutive propositions, we prove the regularities of 5 candidates, 
all of whose corresponding quadratic forms are not regular(see \cite{JP}).

\begin{prop} \label{prop149}
The ternary triangular form $\D(1,4,9)$ is regular.
\end{prop}
\begin{proof}
Let $L=\langle 1,4,9 \rangle$ be a ternary quadratic form and let $\ell=8n+14$ be an integer such that $\ell \longrightarrow \gen(L)$. 
One may easily check that $R(\ell ,L)=R_{(1,1,1)}(\ell ,L)$. 
Thus it suffices to show that $\ell \longrightarrow L$. 
Since
$$
\gen(L)=\left\{ L,K=\langle 1,1,36\rangle \right\},
$$
we may assume that $\ell \longrightarrow K$.

First, assume that $\ell\equiv 0,1\Mod 3$. 
Since $\ell \longrightarrow K$, there is a vector $(x,y,z)\in \z^3$ such that $x^2+y^2+36z^2=\ell$. 
We have $x\equiv 0\Mod 3$ or $y\equiv 0\Mod 3$ and thus $\ell \longrightarrow \langle 1,9,36\rangle \longrightarrow L$.

Now, assume that $\ell \equiv 2\Mod 3$. 
We assert that there is a vector $(x_1,y_1,z_1)\in R(\ell,K)$ such that $x_1\not\equiv \pm y_1\Mod 9$ or $z_1\not\equiv 0\Mod 3$.
Assume to the contrary that there is no such vector. 
Then, we may assume that there is a vector $(u,v,w)\in R(\ell,K)$ such that $u\equiv v\Mod 9$ and $w\equiv 0\Mod 3$.
Let
$$
T=\frac19 \begin{pmatrix} 3&6&36\\6&3&-36\\-1&1&-3 \end{pmatrix}.
$$
Note that 
$$
M_K=\begin{pmatrix} 1&0&0\\0&1&0\\0&0&36 \end{pmatrix}\quad \text{and} \quad
{}^tTM_KT=M_K.
$$
If we let 
$$
\begin{pmatrix} u_1 \\ v_1 \\ w_1 \end{pmatrix}
=T \begin{pmatrix} u \\ v \\ w \end{pmatrix},
$$
then one may check that $(u_1,v_1,w_1)\in \z^3$ and thus $(u_1,v_1,w_1)\in R(\ell,K)$. 
Thus $u_1\equiv \pm v_1\Mod 9$ and $w_1\equiv 0\Mod 3$ by assumption.
Since
$$
u_1-v_1=\frac{-u+v}{3}+8w\equiv 0\Mod 3,
$$
we have $u_1\equiv v_1\Mod 9$. 
From this, one may easily check that $T$ satisfies all conditions given in Lemma \ref{lemeigen} with $S=M_K$, and thus we have $(u,v,w)\in \text{ker}(T-I)$. 
Since $\text{ker}(T-I)=\langle (1,1,0)\rangle$, we have
$(u,v,w)=k(1,1,0)$ for some integer $k$ and $u^2+v^2+36w^2=2k^2$.
This is a contradiction to the fact that $\ell \equiv 6\Mod 8$, and we may conclude
that there is a vector $(x_2,y_2,z_2)\in R(\ell,K)$ such that 
$$
x_2\not\equiv \pm y_2\Mod 9 \ \  \text{or}  \ \ z_2\not\equiv 0\Mod 3.
$$ 
By changing signs of $x_2,y_2,z_2$ and by interchanging the role of $x_2$ and $y_2$, if necessary, 
 we may assume that there is a vector $(x_3,y_3,z_3)\in R(\ell,K)$ such that 
$2x_3+y_3+12z_3\equiv 0\Mod 9$. 
If we let 
$$
(x_4,y_4,z_4)=\left( \displaystyle\frac{x_3+2y_3-12z_3}{3},\frac{x_3-y_3-3z_3}{3},\frac{2x_3+y_3+12z_3}{9} \right),
$$
then one may easily show that 
$(x_4,y_4,z_4)\in R\left(\ell,L\right)$. This completes the proof. 
\end{proof}

\begin{prop}
The ternary triangular form $\D(1,3,27)$ is regular.
\end{prop}
\begin{proof}
Let $L=\langle 1,3,27 \rangle$ be a ternary quadratic form and let $\ell=8n+31$ be an integer such that $\ell \longrightarrow \gen(L)$. 
Note that
$$
\gen(L)=\left\{ L,K= \langle 3 \rangle \perp 
\begin{pmatrix} 4 & 1 \\ 1 & 7 \end{pmatrix} \right\}.
$$
By \cite[Theorem 2.3]{O1} one may show that any integer congruent to 7 modulo 8 that is represented by $K$ is also represented by $L$. Therefore, $\ell$ is represented by $L$.  Note that if 
$x^2+3y^2+27z^2=\ell$, then 
$$
(x^2,3y^2,27z^2)\equiv (1,3,3),(0,4,3),(4,0,3),(0,3,4)\ \ \text{or}\ \ (4,3,0) \Mod 8.
$$
Therefore, if there is a vector $(x,y,z)\in R(\ell,L)$ with $x\equiv y\Mod 2$, then we are done by Lemma \ref{lem13}. 
Thus we may assume that for any $(x,y,z)\in R(\ell,L)$,
$$
y\equiv 1\Mod 2,\ x\equiv z\equiv 0\Mod 2 \ \ \text{and} \ \ x\not\equiv z\Mod 4.
$$

Suppose  that $xy\not\equiv 0\Mod 3$ for any $(x,y,z)\in R\left(\ell,L\right)$.
Let $(u,v,w)\in R\left(\ell,L\right)$ with $u\equiv v\Mod 3$. 
For a rational isometry 
$$
T=\frac{1}{12} \begin{pmatrix} -3 & 18 & -27 \\ 6 & 0 & -18 \\ 1 & 2 & 9 \end{pmatrix},
$$
of $M_L$, we apply Lemma \ref{lemeigen}. Then we have $(u,v,w)\in \text{ker}(T+I)$.
Since $\text{ker}(T+I)=\langle (2,-1,0) \rangle$, we have $(u,v,w)= k(2,-1,0)$ for some integer $k$. One may easily check that $\vert k\vert >1$ and $(k,6)=1$. 
Hence there is a prime $q\ge 5$ such that $k=qs$ and $s\in \z$. Then 
\begin{equation} \label{eq1327}
\ell=u^2+3v^2+27w^2=7q^2s^2.
\end{equation}
On the other hand,
$$
r_{(1,1,1)}\left(\ell,L\right)=\frac23 r\left(\ell,(y-2x)^2+3y^2+27z^2\right)
=\frac23 r\left(\ell,\langle 27 \rangle \perp \begin{pmatrix} 4 & 2 \\ 2 & 4  \end{pmatrix} \right).
$$
If we let $M_1=\langle 27 \rangle \perp \begin{pmatrix} 4 & 2 \\ 2 & 4  \end{pmatrix}$, then
$$
\gen(M_1)=\left\{ M_1,M_2=\begin{pmatrix} 7&1&1\\1&7&1\\1&1&7 \end{pmatrix},
M_3=\langle 3 \rangle \perp \begin{pmatrix} 4&2\\2&28 \end{pmatrix} \right\}, \quad
\spn(M_1)=\{ M_1,M_2 \}.
$$
Note that $7\longrightarrow M_2$. By \cite[Proposition 1]{BH},
we have $7q^2\longrightarrow M_1$ and thus $\ell=7q^2s^2\longrightarrow M_1$.
Thus $r_{(1,1,1)}(\ell,L)>0$ and we are done with this case.

Now, suppose that there is a vector 
$(x_1,y_1,z_1)\in R\left(\ell,L\right)$ such that $x_1y_1\equiv 0 \Mod 3$. 
We define
$$
(x_2,y_2,z_2)=\begin{cases}\displaystyle \left(\frac{x_1+9z_1}{2},y_1,\frac{-x_1+3z_1}{6}\right) \quad \text{if $x_1\equiv 0\Mod 3$},\\[1.2em]
\displaystyle \left( \frac{x_1+9z_1}{2},\frac{-x_1+3z_1}{2},\frac{y_1}{3} \right) \quad \text{otherwise}.
\end{cases}
$$
Then, one may easily check that $(x_2,y_2,z_2)\in R_{(1,1,1)}\left(\ell,L\right)$. 
\end{proof}

\begin{prop}
The ternary triangular form $\D(1,6,27)$ is regular.
\end{prop}
\begin{proof}
Let $L=\langle 1,6,27 \rangle$ be a ternary quadratic form and let $\ell=8n+34$ be an integer such that $\ell \longrightarrow \gen(L)$. 
Note that
$$
\gen(L)=\left\{ L,K=\langle 6 \rangle \perp \begin{pmatrix} 4&1\\1&7 \end{pmatrix} \right\}.
$$
Since $\lambda_2(L)\simeq \langle 3 \rangle \perp 
\begin{pmatrix} 2&1\\1&14 \end{pmatrix} \simeq \lambda_2(K)$, 
we have 
$$
r(\ell,L)=r\left( \frac{\ell}2, \langle 3 \rangle \perp 
\begin{pmatrix} 2&1\\1&14 \end{pmatrix}\right) =r(\ell,M)
$$
and thus $\ell \longrightarrow L$.
If $(x,y,z)\in R\left(\ell,L\right)$, then
$$
(x^2,6y^2,27z^2)\equiv (0,6,4),(4,6,0)\text{ or } (1,6,3) \Mod 8.
$$
Thus we may assume that for any $(x,y,z)\in R\left(\ell,L\right)$, 
$$
y\equiv 1\Mod 2,\ \ x\equiv z\equiv 0\Mod 2,\ \ \text{and}\ \ x\not\equiv z\Mod 4.
$$

First, suppose that there is a vector $(x_1,y_1,z_1)\in R(\ell,L)$ with $x_1\equiv 0\Mod 3$.
If we let 
$$
(x_2,y_2,z_2)=\left(\frac{x_1+9z_1}{2},y_1,\frac{-x_1+3z_1}{6} \right),
$$
then one may easily check that $(x_2,y_2,z_2)\in R_{(1,1,1)}(\ell,L)$. 
Hence we may further assume that for any $(x,y,z)\in R(\ell,L)$, $x\not\equiv 0\Mod 3$. 

Now, suppose that there is a vector $(x_3,y_3,z_3)\in R(\ell,L)$ with 
$y_3\equiv 0\Mod 3$. Let $y_3=3y_3'$. 
Then we have $x_3^2+27(2y_3'^2+z_3^2)=\ell$.
Since $y_3'\equiv 1\Mod 2$, we have $2y_3'^2+z_3^2\neq 0$. 
By Lemma \ref{lem12}, there is a vector $(x_4,y_4,z_4)\in \z^3$ with $y_4\not\equiv z_4\Mod 3$ such that 
$x_4^2+27(2y_4^2+z_4^2)=\ell$. Thus $(x_4,3y_4,z_4)\in R(\ell,L)$ such that 
$y_4\not\equiv 0\Mod 3$ or $z_4\not\equiv 0\Mod 3$. By changing signs of
$x_4,y_4,z_4$, if necessary, we may assume that $x_4\equiv y_4+z_4 \Mod 3$.
If we let
$$
(x_5,y_5,z_5)=\left( \frac{x_4+12y_4+3z_4}{2},\frac{-3y_4+6z_4}{3},\frac{-3x_4+12y_4+3z_4}{18} \right),
$$
then one may easily check that $(x_5,y_5,z_5)\in R_{(1,1,1)}(\ell,L)$. 
Therefore, we further assume that for any $(x,y,z)\in R(\ell,L)$, $xy\not\equiv 0\Mod 3$.

Suppose that there is a vector $(x_6,y_6,z_6)\in R(\ell,L)$ such that 
$y_6\not\equiv \pm 4x_6\Mod 9$ or $z_6\not\equiv 0\Mod 3$. 
Then one may check that by changing signs of $x_6,y_6,z_6$, if necessary, we may assume that
$$
x_6+y_6-3z_6\equiv 0\Mod 9 \text{ or } x_6-4y_6-3z_6\equiv 0\Mod 9.
$$
If $x_6+y_6-3z_6\equiv 0\Mod 9$, then we define
$$
(x_7,y_7,z_7)=\left( \frac{x_6+9z_6}{2},\frac{-x_6-y_6+3z_6}{3},\frac{-x_6+8y_6+3z_6}{18} \right).
$$
If $x_6-4y_6-3z_6\equiv 0\Mod 9$, then we define
$$
(x_7,y_7,z_7)=\left( \frac{x_6+4y_6+3z_6}{2},\frac{-x_6+y_6+3z_6}{3},\frac{x_6-4y_6+15z_6}{18} \right).
$$
Then one may easily check that $(x_7,y_7,z_7)\in R_{(1,1,1)}(\ell,L)$ in each case. 
Now, we further assume that for any $(x,y,z)\in R(\ell,L)$,
\begin{equation} \label{assump4}
y\equiv \pm 4x\Mod 9\ \ \text{and} \ \ z\equiv 0\Mod 3.
\end{equation}

Suppose that there is a vector $(x_8,y_8,z_8)\in R(\ell,L)$ such that $z_8\not\equiv 0\Mod 9$. 
By changing signs of $y_8$ and $z_8$, if necessary, we may assume that $y_8\equiv 4x_8\Mod 9$ and 
$\displaystyle\frac{x_8-y_8}{3}+z_8\not\equiv \pm 4x_8\Mod 9$. 
If we let
$$
(x_9,y_9,z_9)=\left( 2y_8+3z_8,\frac{x_8-y_8+3z_8}{3},\frac{-x_8-2y_8+6z_8}{9} \right),
$$
then $(x_9,y_9,z_9)\in R(\ell,L)$ and $y_9\not\equiv \pm 4x_9\Mod 9$. This contradicts to our assumption \eqref{assump4}.
Therefore, we further assume that for any $(x,y,z)\in R(\ell,L)$, 
$$
y\equiv \pm 4x\Mod 9\ \ \text{and}\ \ z\equiv 0\Mod 9.
$$ 
Take a vector $(u,v,w)\in R(\ell,L)$ with $u\equiv v\Mod 3$ so that $u+2v+6w\equiv 0\Mod 9$. 
If we let
$$
T=\frac19 \begin{pmatrix} 0&18&-27\\3&-3&-9\\1&2&6 \end{pmatrix},
$$
then one may easily check that
$$
M_L=\begin{pmatrix} 1&0&0\\0&6&0\\0&0&27 \end{pmatrix}\quad \text{and}\quad {}^tTM_LT=M_L.
$$
If we let
$$
\begin{pmatrix} u_1 \\ v_1 \\ w_1 \end{pmatrix}
=T\begin{pmatrix} u \\ v \\ w \end{pmatrix},
$$
then clearly, $(u_1,v_1,w_1)\in \z^3$, and thus 
$(u_1,v_1,w_1)\in R(\ell,L)$. 
Note that $u_1-v_1\equiv 0\Mod 3$. 
From this, one may show that $T$ satisfies all conditions given in Lemma \ref{lemeigen} with $S=M_L$, and thus we have $(u,v,w)\in \text{ker}(T+I)$. 
Since $\text{ker}(T+I)=\langle (2,-1,0) \rangle$,
we have $(u,v,w)=k(2,-1,0)$ for some integer $k$ with $\vert k \vert >1$ and $(k,6)=1$. 
Thus there is a prime divisor $q\ge 5$ of $k$. Now $\ell=10q^2s^2$ for some odd integer $s$. 
Note that
$$
r_{(1,1,1)}(\ell,L)=2r\left(\ell,(z-4x)^2+6y^2+27z^2\right)=2r\left( \ell,\langle 6 \rangle \perp \begin{pmatrix} 16&4\\4&28 \end{pmatrix} \right).
$$
Let $M_1=\langle 6 \rangle \perp \begin{pmatrix} 16&4\\4&28 \end{pmatrix}$.
Then 
$$
\gen(M_1)=\spn(M_1)=\left\{ M_1,M_2=\langle 4,6,108 \rangle \right\}.
$$ 
Note that $10\longrightarrow M_2$. By \cite[Proposition 1]{BH}, we have
$r(10q^2s^2,M_1)>0$, and this completes the proof.
\end{proof}

\begin{prop} \label{prop1918}
The ternary triangular form $\D(1,9,18)$ is regular.
\end{prop}
\begin{proof}
Let $L=\langle 1,9,18 \rangle$ be a ternary quadratic form and let $\ell =8n+28$ be an integer such that $\ell \longrightarrow \gen(L)$. 
Note that
$$
\gen (L)=\left\{ L,K=\begin{pmatrix} 4 & 1 & -1 \\ 1 & 7 & 2 \\ -1 & 2 & 7
\end{pmatrix} \right\} .
$$
Since $\lambda_2(L) \simeq \langle 9 \rangle \perp \begin{pmatrix} 2 & 1 \\ 
1 & 5 \end{pmatrix} \simeq \lambda_2(K)$, we have $\ell \longrightarrow L$.
Let $(x,y,z)\in R(\ell,L)$. We may assume that 
$x\equiv y\equiv z\equiv 0\Mod 2$. 
Then $x\not\equiv y\Mod 4$. 

First, assume that $x \not \equiv 0 \Mod 3$ and $y^2+2z^2>0$.
Then by Lemma \ref{lem12}, 
there is a vector $(y_1,z_1)\in \z^2$ with $y_1\not\equiv z_1 \Mod 3$,\ $y_1\equiv y\Mod 4$ and $z_1\equiv z\Mod 2$ such that $y_1^2+2z_1^2=y^2+2z^2$.
So $x^2+9y_1^2+18z_1^2=\ell$. 
By replacing $x$ by $-x$, if necessary, we may assume $x+y_1-z_1\equiv 0 \Mod 3$. 
If we let 
$$
(x_2,y_2,z_2)=\left( \displaystyle\frac{3x+9y_1+18z_1}{6},\frac{-x+5y_1-2z_1}{6},
\frac{-x-y_1+4z_1}{6} \right),
$$ 
then one may easily check that $(x_2,y_2,z_2)\in R_{(1,1,1)}\left( \ell,L\right)$.

Now, assume that $x \not \equiv 0 \Mod 3$ and $y=z=0$.
Note that
$$
r_{(1,1,1)}\left(\ell,L\right)=2r\left( \ell,(v-4u)^2+9v^2+18w^2\right)
=2r\left(\frac{\ell}{2}, \langle 9 \rangle \perp \begin{pmatrix} 5 & 2 \\ 
2 & 8 \end{pmatrix} \right).
$$
If we let $M_1=\langle 9 \rangle \perp \begin{pmatrix} 5 & 2 \\ 
2 & 8 \end{pmatrix}$, then 
$$
\gen(M_1)=\spn(M_1)=\left\{ M_1, M_2=\langle 36 \rangle \perp 
\begin{pmatrix} 2 & 1 \\ 1 & 5 \end{pmatrix} \right\} .
$$
Then by \cite[Proposition 1]{BH}, $2p^2\longrightarrow M_1$ for any prime $p\ge 5$. 
Note that
$$
\frac{\ell}{2}=2\left(\frac{x}{2}\right)^2,\quad \left( \frac{x}{2},6\right )=1\quad \text{and}\quad \frac{x}{2}>1.
$$
So there is a prime divisor $q$ of $\displaystyle\frac{x}{2}$ with $q\ge 5$. 
Thus we have $r\left( \displaystyle\frac{\ell}{2},M_1\right) >0$.

Finally, assume that $x\equiv 0 \Mod 3$.
If we let
$$
(x_3,y_3,z_3)=\left(\displaystyle\frac{3x+9y+18z}{6},\frac{-x-3y+6z}{6},
\frac{-x+3y}{6} \right),
$$
then one may easily check that 
$(x_3,y_3,z_3)\in R_{(1,1,1)}(\ell,L)$. 
\end{proof}

\begin{prop} \label{prop1118}
The ternary triangular form $\D(1,1,18)$ is regular.
\end{prop}
\begin{proof}
Let $L=\langle 1,1,18 \rangle$ be a ternary quadratic form and let $\ell =8n+20$ be an integer such that $\ell \longrightarrow \gen(L)$. Note that
$$
\gen (L)=\left\{ L,K=\langle 2 \rangle \perp \begin{pmatrix} 2 & 1 \\ 1 & 5 \end{pmatrix} \right\} .
$$
Since $\lambda_2(L) \simeq \langle 1,1,9 \rangle \simeq \lambda_2(K)$, we have $\ell \longrightarrow L$.
Let $(x,y,z)\in R(\ell,L)$. 

First, assume that $\ell \equiv 0 \Mod 3$. 
Then $x\equiv y\equiv 0\Mod 3$ and thus $\ell \equiv 0 \Mod 9$. 
So 
$$
\left(\displaystyle\frac{x}{3}\right)^2+\left(\displaystyle\frac{y}{3}\right)^2+2z^2=\displaystyle\frac{\ell}{9}.
$$
Note that $\displaystyle\frac{\ell}{9} \ge 4$ and $\displaystyle\frac{\ell}{9} \equiv 4 \Mod 8$. 
Since the triangular form $\Delta (1,1,2)$ is universal, there is a vector $(x_1,y_1,z_1)\in R_{(1,1,1)}\left( \displaystyle\frac{\ell}{9},\langle 1,1,2 \rangle \right)$
and thus $(3x_1,3y_1,z_1)\in R_{(1,1,1)}\left( \ell,L\right)$.

Now, assume $\ell \equiv 1\Mod 3$. Note that $xy\equiv 0\Mod 3$. 
Without loss of generality, we may assume that $y\equiv 0\Mod 3$. 
Then 
$$
\ell=x^2+9\left(\displaystyle\frac{y}{3} \right)^2+18z^2.
$$
Note that $\ell \ge 28$, $\ell \equiv 4\Mod 8$. 
Since $\Delta(1,9,18)$ is regular by Proposition \ref{prop1918}, 
there is a vector $(x_2,y_2,z_2)\in R_{(1,1,1)}\left( \ell,\langle 1,9,18 \rangle \right)$ and thus 
$(x_2,3y_2,z_2)\in R_{(1,1,1)}\left( \ell,\langle 1,1,18 \rangle \right)$.

Finally, assume that $\ell \equiv 2\Mod 3$. Since $x^2+y^2+18z^2\equiv 4\Mod 8$, 
we may assume that $x\equiv 0\Mod 4$, $y\equiv 2\Mod 4$ and $z\equiv 0\Mod 2$. 
Since $xy\not\equiv 0\Mod 3$, we may further assume that $x\equiv y\Mod 3$. If we let
$$
(x_3,y_3,z_3)=\left(\frac{x+y}{2}+3z,-\frac{x+y}{2}+3z,\frac{-x+y}{6} \right),
$$
then one may easily check that $(x_3,y_3,z_3)\in R_{(1,1,1)}\left( \ell,L\right)$. 
\end{proof}

\begin{table}[ht]
\caption{Regular ternary triangular forms}
\begin{tabular}{|lll|}
\hline \rule[-2mm]{0mm}{6mm}
$\Delta_1=\Delta(1,1,1)$,\quad \quad \quad \quad &$\Delta_2=\Delta(1,1,2)$,\quad \quad \quad \quad &$\Delta_3=\Delta(1,1,3)$,\\
\hline \rule[-2mm]{0mm}{6mm} 
$\Delta_4=\Delta(1,1,4)$, &$\Delta_5=\Delta(1,1,5)$,\ &$\Delta_6=\Delta(1,1,6)$,\\
\hline \rule[-2mm]{0mm}{6mm}
$\Delta_7=\Delta(1,2,2)$, &$\Delta_8=\Delta(1,2,3)$, &$\Delta_9=\Delta(1,2,4)$,\\
\hline \rule[-2mm]{0mm}{6mm}
$\Delta_{10}=\Delta(1,1,9)$,&
$\Delta_{11}=\Delta(1,3,3)$,&
$\Delta_{12}=\Delta(1,2,5)$,\\
\hline \rule[-2mm]{0mm}{6mm}
$\Delta_{13}=\Delta(1,1,12)$,&
$\Delta_{14}=\Delta(1,3,4)$,&
$\Delta_{15}=\Delta(2,2,3)$,\\
\hline \rule[-2mm]{0mm}{6mm}
$\Delta_{16}=\Delta(1,1,18)$,&
$\Delta_{17}=\Delta(1,3,6)$,&
$\Delta_{18}=\Delta(2,3,3)$,\\
\hline \rule[-2mm]{0mm}{6mm}
$\Delta_{19}=\Delta(1,2,10)$,&
$\Delta_{20}=\Delta(1,1,21)$,&
$\Delta_{21}=\Delta(1,4,6)$,\\
\hline \rule[-2mm]{0mm}{6mm}
$\Delta_{22}=\Delta(1,5,5)$,&
$\Delta_{23}=\Delta(1,3,9)$,&
$\Delta_{24}=\Delta(1,3,10)$,\\
\hline \rule[-2mm]{0mm}{6mm}
$\Delta_{25}=\Delta(1,3,12)$,&
$\Delta_{26}=\Delta(1,4,9)$,&
$\Delta_{27}=\Delta(1,6,6)$,\\
\hline \rule[-2mm]{0mm}{6mm}
$\Delta_{28}=\Delta(3,3,4)$,&
$\Delta_{29}=\Delta(1,5,10)$,&
$\Delta_{30}=\Delta(1,3,18)$,\\
\hline \rule[-2mm]{0mm}{6mm}
$\Delta_{31}=\Delta(1,6,9)$,&
$\Delta_{32}=\Delta(2,3,9)$,&
$\Delta_{33}=\Delta(3,3,7)$,\\
\hline \rule[-2mm]{0mm}{6mm}
$\Delta_{34}=\Delta(2,3,12)$,&
$\Delta_{35}=\Delta(1,3,27)$,&
$\Delta_{36}=\Delta(1,9,9)$,\\
\hline \rule[-2mm]{0mm}{6mm}
$\Delta_{37}=\Delta(1,3,30)$,&
$\Delta_{38}=\Delta(2,5,10)$,&
$\Delta_{39}=\Delta(1,9,12)$,\\
\hline \rule[-2mm]{0mm}{6mm}
$\Delta_{40}=\Delta(2,3,18)$,&
$\Delta_{41}=\Delta(1,5,25)$,&
$\Delta_{42}=\Delta(3,7,7)$,\\
\hline \rule[-2mm]{0mm}{6mm}
$\Delta_{43}=\Delta(2,5,15)$,&
$\Delta_{44}=\Delta(1,6,27)$,&
$\Delta_{45}=\Delta(1,9,18)$,\\
\hline \rule[-2mm]{0mm}{6mm}
$\Delta_{46}=\Delta(1,9,21)$,&
$\Delta_{47}=\Delta(1,21,21)$,&
$\Delta_{48}=\Delta(5,6,15)$,\\
\hline \rule[-2mm]{0mm}{6mm}
$\Delta_{49}=\Delta(3,7,63)$.&&\\
\hline
\end{tabular}
\end{table}

\begin{thm} \label{mainthm}
There are exactly $49$ regular ternary triangular forms, which are listed in Table 4.
\end{thm}

\begin{proof}
For $1\le i\le 49$, we write $\Delta_i=\Delta(a_i,b_i,c_i)$. Let $L_i=\langle a_i,b_i,c_i\rangle$ be a ternary quadratic form and let $\ell_i(n)=8n+a_i+b_i+c_i$ be any integer such that $\ell_i(n)\longrightarrow \gen(L_i)$. 
In Theorem \ref{thmfloor} and Propositions \ref{prop149}$\sim$\ref{prop1118}, we have already proved the regularity of each $\Delta_i$ when 
$$
i\in \{ k: 1\le k\le 9, 12\le k\le 16,\ \text{or} \ k=19,20,21,24,26,35,44,45 \}.
$$ 
Hence we may assume that $i$ is not contained in the above set. 
Note that for any integer $i$ which is not contained in $\{ 16,26,35,44,45\}$, which we alreay considered in Propositions \ref{prop149}$\sim$\ref{prop1118},
the corresponding quadratic form $L_i$ has class number $1$ and thus $\ell_i(n) \longrightarrow L_i$. 
If $i\in \{ 10,36,39,40,41,49\}$,
then one may easily show that $R(\ell_i(n),L_i)=R_{(1,1,1)}(\ell_i(n),L_i)$. 
Hence $\ell_i(n)\overset{2}{\longrightarrow} L_i$ in this case. 

Now, we consider the case when $i=30$. Note that if $x^2+3y^2+18z^2=8n+22$, then we have
$z\equiv 1\Mod 2$ and $x\equiv y\Mod 2$. By Lemma \ref{lem13}, we have
$$
\begin{array}{rl}
r_{(1,1,1)}(8n+22,\langle 1,3,18 \rangle)&=\displaystyle\sum_{z\in \z}
r_{(1,1)}(8n+22-18z^2,\langle 1,3 \rangle) \\[1.2em]
&=\displaystyle\frac23 r(8n+22,\langle 1,3,18 \rangle).
\end{array}
$$
Since the proof of the case when $i=48$ is quite similar to this, we omit the proof.

Assume  that $i=31$. Since the quadratic form $\langle 1,6,9 \rangle$ has class number $1$
and it primitively represents $8n+16$ over $\z_2$, there is a vector
$$
(x,y,z)\in R(8n+16,\langle 1,6,9\rangle),\quad (x,y,z,2)=1.
$$
Since $x^2+6y^2+9z^2\equiv 0\Mod 8$, we have $xyz\equiv 1\Mod 2$.

For the remaining $i$, that is, 
$$
i\in \{ 11,17,18,22,23,25,27,28,29,32,33,34,37,38,42,43,46,47 \},
$$
one may check that $\Delta(a_i,b_i,c_i)$ can be obtained from a ternary triangular form whose regularity is already proved
by taking $\lambda_p$-transformations several times for some $p\in \{3,5,7\}$.  Furthermore, one may easily check that the regularity is preserved during taking the $\lambda_p$-transformation.  This completes the proof.
\end{proof}


\end{document}